\documentclass[11pt]{article}
\usepackage[utf8]{inputenc}

\usepackage{geometry} 
\usepackage{graphicx} 

\usepackage{booktabs} 
\usepackage{array} 
\usepackage{paralist} 
\usepackage{verbatim} 
\usepackage{subfig} 
\usepackage{hyperref}
\usepackage{amsmath}
\usepackage{amssymb}
\usepackage{authblk}
\usepackage{stmaryrd}
\usepackage{enumitem}
\setlist{nosep}
\usepackage{bbm}
\usepackage{mathtools}
\usepackage{textcomp}
\usepackage{xcolor}
\usepackage{listings}
\usepackage{float}
\usepackage{centernot}
\usepackage[mathscr]{euscript}
\usepackage{tikz-cd}

\newcommand{\Bcal}{\mathcal{B}}
\newcommand{\Ccal}{\mathcal{C}}
\newcommand{\Dcal}{\mathcal{D}}
\newcommand{\Ecal}{\mathcal{E}}
\newcommand{\Fcal}{\mathcal{F}}

\newcommand{\Scal}{\mathcal{S}}

\newcommand{\Nbb}{\mathbb{N}}
\newcommand{\Qbb}{\mathbb{Q}}
\newcommand{\Rbb}{\mathbb{R}}

\newcommand{\Zbb}{\mathbb{Z}}

\newcommand{\yon}{\mathbf{y}}
\newcommand{\Set}{\mathbf{Set}}
\newcommand{\Setswith}[1]{\PSh(#1)}

\newcommand{\Hom}{\mathrm{Hom}}

\newcommand{\op}{^{\mathrm{op}}}

\newcommand{\ns}{\mathrm{ns}}

\newcommand{\lr}[2]{( #1, #2 )}
\newcommand{\fllr}[2]{[ #1, #2 )}

\newcommand{\too}{\twoheadrightarrow}

\newcommand{\et}{\mathrm{\acute{e}t}}

\DeclareMathOperator{\ob}{ob}
\DeclareMathOperator{\id}{id}

\DeclareMathOperator{\Fix}{Fix}
\DeclareMathOperator{\PSh}{\mathbf{PSh}}

\DeclareMathOperator{\Sh}{\mathbf{Sh}}

\DeclareMathOperator{\Conn}{Conn}
\DeclareMathOperator{\End}{End}

\tikzset{
  no line/.style={draw=none,
    commutative diagrams/every label/.append style={/tikz/auto=false}},
  from/.style args={#1 to #2}{to path={(#1)--(#2)\tikztonodes}}
	}
\tikzset{symbol/.style={draw=none, every to/.append style={edge node = {node [sloped, allow upside down, auto=false] {$#1$}}}}}

\usepackage{amsthm}
\newtheorem{thm}{Theorem}[section]
\newtheorem{theorem}[thm]{Theorem}
\newtheorem{prop}[thm]{Proposition}
\newtheorem{proposition}[thm]{Proposition}
\newtheorem{lemma}[thm]{Lemma}
\newtheorem{crly}[thm]{Corollary}
\newtheorem{corollary}[thm]{Corollary}

\newtheorem{scholium}[thm]{Scholium}

\theoremstyle{definition}
\newtheorem{dfn}[thm]{Definition}
\newtheorem{definition}[thm]{Definition}

\newtheorem{example}[thm]{Example}

\theoremstyle{remark}
\newtheorem{rmk}[thm]{Remark}
\newtheorem{remark}[thm]{Remark}

\usepackage{fancyhdr} 
\usepackage[nottoc,notlof,notlot]{tocbibind} 
\usepackage[titles,subfigure]{tocloft} 

\usepackage[nodayofweek]{datetime}

\date{}

\AtBeginDocument{%
   \def\MR#1{}
}

\let\theta\vartheta

\title{Geometric morphisms between toposes of monoid actions: factorization systems}
\author{Jens Hemelaer \thanks{Department of Mathematics, University of Antwerp, 
 Middelheimlaan 1, B-2020 Antwerp (Belgium) \\ email: jens.hemelaer@uantwerpen.be} \\ Morgan Rogers \thanks{Universit\'e Sorbonne Paris Nord, LIPN - UMR7030 CNRS,\\
 99 Avenue Jean-Baptiste Cl\'ement, 93430, Villetaneuse \\ email: rogers@lipn.univ-paris13.fr}}

\begin{document}

\maketitle

\abstract{Let $M,\,N$ be monoids, and $\Setswith{M},\,\Setswith{N}$ their respective categories of right actions on sets. In this paper, we systematically investigate correspondences between properties of geometric morphisms $\PSh(M) \to \PSh(N)$ and properties of the semigroup homomorphisms $M \to N$ or flat-left-$N$-right-$M$-sets inducing them. More specifically, we consider properties of geometric morphisms featuring in factorization systems, namely: surjections, inclusions, localic morphisms, hyperconnected morphisms, terminal-connected morphisms, \'etale morphisms, pure morphisms and complete spreads. We end with an application to topos-theoretic Galois theory to the special case of toposes of the form $\PSh(M)$.}

\tableofcontents

\section{Introduction}
\label{sec:intro}

This article is part of our ongoing project in which we study toposes of presheaves $\PSh(M)$ with $M$ a monoid, and the geometric morphisms between these toposes. A topos of this form appeared in the construction of the Arithmetic Site of Connes and Consani \cite{connes-consani}, in the special case where $M$ is the monoid of nonzero natural numbers under multiplication. Variations on the Arithmetic Site, with different choices of monoid $M$, were considered in \cite{sagnier}, \cite{arithmtop} and \cite{llb-three}. Further, if we think of a commutative monoid $M$ as dual to an ``affine $\mathbb{F}_1$-scheme'', as in \cite{manin}, then $\PSh(M)$ can be seen as a category of quasi-coherent modules on such an $\mathbb{F}_1$-scheme, see \cite{pirashvili}. 

In semigroup theory, studying the topos $\PSh(M)$ can give a helpful alternative point of view: in \cite{hr1} it was demonstrated that various known facts from semigroup theory have natural topos-theoretic interpretations. In \cite{hr-fitzgerald} a problem in semigroup theory was solved by the present authors with the help of topos-theoretic language; conversely, in \cite{not-locally-connected} a geometric morphism between toposes of this form provided a counterexample to an open question in topos theory.

In \cite{hr1}, we restricted our attention to the study of the global section geometric morphism $\PSh(M) \to \Set$. This time, we will look at more general geometric morphisms $\PSh(M) \to \PSh(N)$, with $M$ and $N$ monoids. Recall that in \cite{TDMA}, the second named author presented a 2-categorical equivalence between a 2-category of discrete monoids and a 2-category whose objects are their (presheaf) toposes of right actions, whereby essential geometric morphisms between the toposes correspond to semigroup homomorphisms between the monoids; the global sections morphism of $\PSh(M)$ corresponds to the unique semigroup homomorphism $M \to 1$, for example. As explained in \cite{hr1}, general geometric morphisms $\Setswith{M} \to \Setswith{N}$ correspond to sets equipped with a flat left $N$-action and a compatible right $M$-action. In this paper, we refer to these as $\fllr{N}{M}$-sets; see Definition \ref{dfn:fllr}. A natural next step in studying toposes of discrete monoid actions is an investigation of how properties of geometric morphisms descend to properties of the corresponding semigroup homomorphisms or $\fllr{N}{M}$-sets. Since properties of geometric morphisms are far too varied to examine exhaustively in a single article, we focus here on \textit{factorization systems}.

The first factorization systems that we will consider are the (surjection, inclusion) factorization and the (hyperconnected, localic) factorization. These are the two most well-known factorization systems for geometric morphisms. For essential geometric morphisms between presheaf toposes, an explicit construction for these two factorizations is given in \cite[\S A4.2 and \S A4.6]{Ele}. If we apply this to the special case of an essential geometric morphism $f: \PSh(M) \to \PSh(N)$ induced by a semigroup homomorphism $\phi: M \to N$, then we get a factorization
\begin{equation*}
\begin{tikzcd}[row sep=small,column sep=huge]
M \ar[r,"{\pi}",two heads] & M/{\sim} \, \ar[r,hook,"{\psi}"] 
& eNe \ar[r,"{\iota}",hook] & N \\
\Setswith{M} \ar[r,"{\text{hyperconnected}}"'] & \Setswith{M/{\sim}} \ar[r,"{\text{localic surj.}}"'] & \Setswith{eNe} \ar[r,"{\mathrm{inclusion}}"'] & \Setswith{N}
\end{tikzcd}
\end{equation*}
where the hyperconnected part is induced by the projection of $M$ onto its image $M/{\sim} = \phi(M)$, the localic surjection part is induced by the inclusion of $M/{\sim}$ in $eNe$ (with $e=\phi(1)$), and the inclusion part is induced by the semigroup inclusion $eNe \subseteq N$. The localic part is the composition of the localic surjection part and the inclusion part, while the surjection part is the composition of the hyperconnected part and the localic surjection part. For a general geometric morphisms $f : \PSh(M) \to \PSh(N)$ given by a $[N,M)$-set $A$, we can also consider the (surjection, inclusion) factorization $\PSh(M) \to \Ecal \to \PSh(N)$, but in this case the intermediate topos is not necessarily again a topos of presheaves. However, we can still give concrete characterizations of when $f$ is surjective, localic or hyperconnected, in terms of the $[N,M)$-set $A$.

Another factorization system that we will discuss is the (terminal-connected, \'etale) factorization, which exists for all essential geometric morphisms, see \cite[\S4.7]{caramello-denseness}. For an essential geometric morphism $f : \PSh(M) \to \PSh(N)$, induced by a semigroup homomorphism $\phi: M \to N$, it follows from the definition that the intermediate topos is again a presheaf topos. We describe the factorization as explicitly as possible, which leads to a characterization of when $f$ is terminal-connected (resp.\ \'etale) in terms of the semigroup homomorphism $\phi$. For a more general geometric morphism $f : \PSh(M) \to \PSh(N)$, induced by a $\fllr{N}{M}$-set $A$, we again give a characterization of terminal-connectedness (in the sense of Osmond \cite[Definition 5.3.3]{osmond-2geometries}). Because \'etale geometric morphisms are always essential, they do not have to be considered separately here. However, note that the (terminal-connected, \'etale) factorization does not always exist for general geometric morphisms.

A last factorization that we will consider is the (pure, complete spread) factorization, as studied extensively by Bunge and Funk, see \cite{bunge-funk-spreads-I}, \cite{bunge-funk-spreads-II} and \cite{bunge-funk}. This factorization exists whenever the domain topos is locally connected, and it is conceptually dual to the (terminal-connected, \'etale) factorization mentioned above. For an essential geometric morphism $f : \PSh(M) \to \PSh(N)$ induced by a semigroup homomorphism $\phi: M \to N$, the factorization is dual in a literal sense: the geometric morphism induced by $\phi$ is pure (resp.\ a complete spread) if and only if the geometric morphism induced by $\phi\op : M\op \to N\op$ is terminal-connected (resp.\ \'etale). For a general geometric morphism $f : \PSh(M) \to \PSh(N)$ given by a $\fllr{N}{M}$-set $A$, we give a characterization of when $f$ is pure. In our setting, $f$ can only be a complete spread if it is essential, so only a study of general pure geometric morphisms is needed here.

It follows from the work of Bunge and Funk \cite[Corollary 7.9]{bunge-funk-spreads-II} that the intersection of \'etale geometric morphisms and complete spreads is in our case given by the \textit{locally constant \'etale} morphisms. These are employed in a topos-theoretic version of Galois theory. As an application of our investigation, we recover the result that the Galois groupoid for a topos of the form $\PSh(M)$ is a group, and is exactly the \textit{groupification} of $M$.

\subsection*{Overview}

In the Section \ref{sec:bg}, we recall how semigroup homomorphisms and biactions of monoids induce geometric morphisms, as well as some basic categorical constructions which we shall need later. We tackle the (surjection, inclusion) and (hyperconnected, localic) factorization systems in Section \ref{sec:si-hl}, the (terminal-connected, \'etale) factorization system in Section \ref{sec:tc-etale} and finally the (pure, complete spread) factorization in Section \ref{sec:pure-complete-spread}. Each of these sections begin with some background on the types of morphism involved, followed by an investigation of the factorization system for essential geometric morphisms coming from semigroup homomorphisms; the end of each section contains an attempt to characterize the biactions producing geometric morphisms in the various classes.

In Section \ref{sec:compare}, we investigate the relationship between the latter two factorization systems, in particular giving examples illustrating the various possible relationships between \'etale morphisms and complete spreads. We apply this in Section \ref{sec:galois} to streamline the application in the Galois theory of our toposes of discrete monoid actions.

\section{Background}
\label{sec:bg}

\subsection{Essential Geometric Morphisms}
\label{ssec:essential}

Let $\Ecal$ and $\Fcal$ be Grothendieck toposes. Recall that a \textit{geometric morphism} $f : \Fcal \to \Ecal$ is by definition an adjunction
\begin{equation*}
\begin{tikzcd}
\Fcal \ar[r, bend right,"{f_*}"'] \ar[r, "\bot", phantom] & \Ecal, \ar[l, bend right,"{f^*}"']
\end{tikzcd}
\end{equation*}
with $f_*$, the \textit{direct image functor}, right adjoint to $f^*$, the \textit{inverse image functor}, where the latter is required to preserve finite limits. We follow the convention that a 2-morphism or \textit{geometric transformation} $f \Rightarrow g$ between geometric morphisms $f,g : \Fcal \to \Ecal$ is a natural transformation $f^* \Rightarrow g^*$.

A geometric morphism is said to be \textbf{essential} if $f^*$ has a left adjoint, denoted $f_!$:
\begin{equation*}
\begin{tikzcd}[column sep=large,row sep=large]
\Fcal \ar[r,bend left=65,"{f_!}"] \ar[r,bend right=55,"{f_*}"'] \ar[r, phantom, shift left=6, "\bot"] \ar[r, phantom, shift right=4, "\bot"] & \Ecal. \ar[l,"{f^*}"', shift left = 0.5]
\end{tikzcd}
\end{equation*}
By the Special Adjoint Functor Theorem, a geometric morphism $(f^* \dashv f_*)$ is essential precisely if $f^*$ preserves not just finite limits but all small limits. Recall that a functor $F: \Ccal \to \Dcal$ induces an essential geometric morphism $f: \Setswith{\Ccal} \to \Setswith{\Dcal}$ whose inverse image functor is precomposition with $F\op$, so that $f_*$ and $f_!$ are right and left Kan extensions along $F\op$, respectively. Conversely, any essential geometric morphism between presheaf toposes is (up to natural isomorphism) induced by some functor $F$ in this way, which is recovered by restricting $f_!$ to the representable presheaves.

From Theorem 6.5 of \cite{TDMA}, we have an equivalence between the 2-category of monoids, semigroup homomorphisms and `conjugations' (monoid elements which commute appropriately with homomorphisms), and the 2-category of the corresponding presheaf toposes, essential geometric morphisms and natural transformations, up to reversing the direction of the conjugations. These presheaf toposes have a great deal more structure than the monoids from which they are constructed, and as such this equivalence give us access to a variety of approaches for examining the subtler properties of monoids and their right actions. We recall from \cite{hr1} that, given a monoid $M$, its topos $\Setswith{M}$ of actions is equipped with a canonical point and a global sections morphism:
\begin{equation} \label{eq:can-point-and-global-sections}
\begin{tikzcd}
\Set \ar[r, bend left=55, "- \times M"] \ar[r, bend right=45, "{\Hom_{\Set}(M{,}-)}"',pos=.55]
\ar[r, phantom, shift left=6, "\bot", near end] \ar[r, phantom, shift right=4, "\bot", near end] &
{\Setswith{M}} \ar[l, "U"', near start] \ar[r, "C", bend left = 50] \ar[r, "\Gamma"', bend right= 40]
\ar[r, symbol = \bot, shift right = 4, near start] \ar[r, symbol = \bot, shift left = 6, near start] &
\Set \ar[l, "\Delta"', near end],
\end{tikzcd}
\end{equation}
where the functors not explicitly specified are:
\begin{itemize}
	\item the forgetful functor $U$ sending a right $M$-set to its underlying set;
	\item the global sections functor $\Gamma$ sending an $M$-set $A$ to its set
\begin{equation*}
	\Fix_M(A)= \Hom_{\Setswith{M}}(1,A)
\end{equation*}
	of fixed points under the action of $M$;
	\item the constant sheaf functor $\Delta$ sending a set $B$ to the same set with trivial $M$-action;
	\item the connected components functor $C$ sending an $M$-set $A$ to its set of components under the action of $M$ (that is, to its quotient under the equivalence relation generated by $a \sim a \cdot m$ for $a \in A$, $m \in M$).
\end{itemize}
It should also be noted that for a set $X$, the $M$-action on $\Hom_{\Set}(M{,}X)$ by $m \in M$ sends $f$ to $(n \mapsto f(mn))$, while the $M$ action on $X \times M$ is by right multiplication on the $M$-component.

These geometric morphisms correspond under the equivalence to the canonical monoid homomorphisms $1 \to M$ and $M \to 1$. More generally, an arbitrary monoid homomorphism $\phi$ gets sent to the essential geometric morphism whose inverse image is restriction of the action along $\phi$. The geometric morphism corresponding to an arbitrary semigroup homomorphism is a little more complicated, and can be most concisely described in terms of a $\fllr{N}{M}$ set and a $\lr{M}{N}$-set; see Lemma \ref{lem:esstensor} below.

\subsection{General Geometric Morphisms}

In previous work \cite[Propositions 1.5 and 1.8]{hr1}, we discussed how more generally a geometric morphism $f: \Setswith{M} \to \Setswith{N}$ can be understood as a tensor--hom adjunction. We recall those results here.

\begin{dfn}
If $X$ is a set equipped with a left $N$-action and a right $M$-action, then we say that the left $N$-action and right $M$-action are \textbf{compatible} if $(n\cdot x)\cdot m = n \cdot (x\cdot m)$ for all $n \in N$, $x \in X$ and $m \in M$. Sets with a compatible left $N$-action and right $M$-action will be called \textbf{$(N,M)$-sets}.\footnote{We read this as `left-$N$-right-$M$-set'.} As homomorphisms between these, we of course consider functions commuting with both actions.
\end{dfn}

Recall that a left $N$-set $A$ is said to be \textbf{flat} if the functor
\[- \otimes_N A: [N\op,\Set] \to \Set\]
preserves finite limits, which is equivalent (see e.g.\ \cite[VII, Theorem 3]{MLM})
to the conditions that
\begin{enumerate}
\item $A$ is non-empty;
\item for elements $b,b' \in A$ there exists $a \in A$ and $n,n' \in N$ with $n \cdot a = b$ and $n' \cdot a = b'$; and
\item whenever $c \in A$ and $n,n' \in N$ with $n \cdot c = n' \cdot c$, there exists $d \in A$, $p \in N$ with $p \cdot d = c$ and $np = n'p$.
\end{enumerate}

\begin{dfn}
\label{dfn:fllr}
We say a $\lr{N}{M}$-set $A$ is \textbf{flat}, or a \textbf{$\fllr{N}{M}$-set}, if it is flat as a left $N$-set. The category of $\fllr{N}{M}$-sets forms a full subcategory of the category of $\lr{N}{M}$-sets.
\end{dfn}

As shown below, the category of geometric morphisms $f: \Setswith{M} \to \Setswith{N}$ is equivalent to the category of $\fllr{N}{M}$-sets. More generally, a (Lawvere) \textbf{distribution} $f : \Fcal \to \Ecal$ between toposes is any adjoint pair $f^* \dashv f_*$, where $f^*$ does not necessarily preserve finite limits, \cite{lawvere-measures} \cite{bunge-funk-spreads-I}. A morphism $f \Rightarrow g$ between distributions $f,g : \Fcal \to \Ecal$ is still a natural transformation $f^* \Rightarrow g^*$. We mention the following special case of Diaconescu's Theorem:

\begin{thm} \label{thm:N-M-sets}
There is an equivalence between the category of geometric morphisms $\Setswith{M} \to \Setswith{N}$ and the category of $\fllr{N}{M}$-sets. More generally, there is an equivalence between the category of distributions $\Setswith{M} \to \Setswith{N}$ and the category of $(N,M)$-sets. 
\end{thm}
\begin{proof}
At the level of objects, the equivalences send an adjunction $f^* \dashv f_*$ to the $\lr{N}{M}$-set $f^*(N)$, which has a right $M$-action by virtue of being an object of $\Setswith{M}$, and a left $N$ action coming from the images of the endomorphisms of $N$ as an object of $\Setswith{N}$, which consist of left multiplication by elements of $N$. Conversely, a $\lr{N}{M}$-set $A$ is sent to the tensor-hom adjunction $(- \otimes_N A) \dashv \Hom_{M}(A,-)$; see \cite[Proposition 1.5]{hr1}.

Given two adjunction $f^* \dashv f_*$ and $g^* \dashv g_*$, a natural transformation $f^* \Rightarrow g^*$ is determined by its component $f^*(N) \to g^*(N)$, which is automatically a right-$M$-set homomorphism; it is also a left-$N$-set homomorphism by naturality with respect to the endomorphisms of $N$. Conversely, a $\lr{N}{M}$-set homomorphism $A \to B$ induces a natural transformation $- \otimes_N A \to - \otimes_N B$ by composition on the second component; commutation with the respective actions ensures that this is well-defined and an $M$-set homomorphism at each object $X$.

Finally, the geometric morphisms $f$ are precisely the distributions such that $f^*$ preserves finite limits, so correspond under this equivalence to the full subcategory of $\fllr{N}{M}$-sets, as required.
\end{proof}

Thus we have an algebraic characterization of arbitrary geometric morphisms between toposes of discrete monoid actions, as well as an alternative perspective on the extra adjunction $(f_! \dashv f^*)$ in an essential geometric morphism $f$. Explicitly, by direct calculation:

\begin{lemma}
\label{lem:esstensor}
Let $f: \Setswith{M} \to \Setswith{N}$ be an essential geometric morphism induced by a semigroup homomorphism $\phi:M \to N$. Then the $\fllr{N}{M}$-set corresponding to $(f^* \dashv f_*)$ is the left ideal $N \phi(1)$ equipped with left $N$-action by multiplication and right $M$-action by multiplication after applying $\phi$. In particular, when $\phi$ is a monoid homomorphism, the $\fllr{N}{M}$-set is simply $N$ equipped with the respective actions.

Meanwhile, the $\lr{M}{N}$-set corresponding to the extra adjunction $(f_! \dashv f^*)$ is the right ideal $\phi(1) N$ of $N$, similarly equipped with respective multiplication actions but with the handedness reversed.
\end{lemma}

The correspondence from Theorem \ref{thm:N-M-sets} is well-behaved with respect to composition of geometric morphisms. 

\begin{lemma}
\label{lem:compose1}
Suppose $g: \Setswith{M} \to \Setswith{L}$ and $f: \Setswith{L} \to \Setswith{N}$ are induced by the $\fllr{L}{M}$-set $B$ and the $\fllr{N}{L}$-set $A$ respectively. Then $f \circ g$ is induced by $A \otimes_L B$ (up to isomorphism). This result extends to all distributions. 
\end{lemma}
\begin{proof}
This is immediate from the fact that $g^*f^*(X) \simeq (X \otimes_N A) \otimes_L B$ and the tensor product is associative up to isomorphism.
\end{proof}

\subsection{Categories of Elements and Slice Toposes}
\label{ssec:categories-of-elements}

Let $\Ccal$ be a small category and let $X: \Ccal\op \to \Set$ be a presheaf on $\Ccal$. Recall that the \textbf{category of elements} of $X$ is the category $\int_{\Ccal} X$ having
\begin{itemize}
\item as objects, pairs $(C,a)$, where $C$ is an object of $\Ccal$ and $a \in X(C)$;
\item as morphisms $(C,a) \to (D,b)$ the morphisms $f:C \to D$ such that $b \cdot f = a$;
\item composition given by composition in $\Ccal$.
\end{itemize}
We shall also need the dual construction: viewing a functor $Y: \Ccal \to \Set$ as a contravariant functor defined on $\Ccal\op$, we define $\int^{\Ccal}Y := \left(\int_{\Ccal\op} Y \right)\op$.

When $\Ccal=M$ is a monoid, a presheaf on $M$ is precisely a right $M$-set. Since there is only one object, we simplify the description of categories of elements by dropping the indexing over the objects. If $M$ is commutative, then $\int_M M$ agrees with the category $C(M)$ appearing in \cite[\S4.1]{connes-consani-affine}.

Categories of elements are useful for studying slice toposes. Recall that for any category $\Ecal$ and object $X$ in $\Ecal$, the \textbf{slice category} $\Ecal/X$ is the category with,
\begin{itemize}
\item as objects the morphisms $f: E \to X$ in $\Ecal$ with codomain $X$;
\item as morphisms from $f : E \to X$ to $f' : E' \to X$ the morphisms $g : E \to E'$ such that $f' \circ g = f$;
\item composition given by composition of morphisms in $\Ecal$. 
\end{itemize}
For a topos $\Ecal$ and an object $X$ in $\Ecal$, the slice category $\mathcal{E}/X$ is again a topos, inheriting all of the required properties from $\Ecal$ (this fact is sometimes called the \textit{fundamental theorem of topos theory}). We shall refer to a topos of the form $\mathcal{E}/X$ for a generic object $X$ as a \textbf{slice of $\Ecal$}. The relation between categories of elements above and slice toposes can be described as follows:
\begin{proposition}
\label{prop:slice}
Consider a category $\Ccal$ and a presheaf $X$ on $\Ccal$. Then there is an equivalence of categories
\begin{equation*}
\Setswith{\Ccal}/X ~\simeq~ \PSh \left(\int_{\Ccal} X\right).
\end{equation*}
\end{proposition}
\begin{proof}
In one direction, an object $g:Y \to X$ on the left hand side is sent to the presheaf $\hat{g}$ on $\int_{\Ccal} X$ sending $(C,x)$ to $g(C)^{-1}(\{x\}) \subseteq Y(C)$. In the opposite direction, a presheaf $G$ on $\int_{\Ccal} X$ is sent to the object $\tilde{G}:Y \to X$, where $Y(C) = \coprod_{x \in X(C)}G(C,x)$ and $\tilde{G}$ sends the elements in each $G(C,x)$ to $x$. We leave the remaining details to the reader; this features as Exercise III.8 in \cite{MLM}.
\end{proof}

\subsection{Idempotent Completion}
\label{ssec:idempotentcomplete}

\begin{dfn}
Recall that a category $\Ccal$ is \textbf{idempotent complete} (also known as \textit{Cauchy complete} or \textit{Karoubi complete}) if every idempotent splits, in the sense that given any idempotent $e:C \to C$ in $\Ccal$, there exist morphisms $r:C \to D$ and $s:D \to C$ with $r \circ s = \id_D$ and $s \circ r = e$, which are automatically unique up to unique isomorphism with $D$.
\end{dfn}

We can construct the idempotent completion of any (small) category. This results in an equivalent category of presheaves; conversely, an idempotent complete category can be recovered up to equivalence from its category of presheaves as the subcategory of indecomposable projective objects. Hence there is a unique idempotent complete category up to equivalence representing any presheaf category.

We recall the following from \cite{TDMA}:
\begin{lemma}
\label{lem:Mcheck}
The idempotent completion $\check{M}$ of a monoid $M$ is given by the category with,
\begin{itemize}
\item as objects the idempotents of $M$ (the object corresponding to an idempotent $e \in M$ is denoted by $\underline{e}$);
\item as morphisms $\underline{e} \to \underline{d}$ the elements $m \in M$ such that $me = m = dm$;
\item composition given by multiplication in $M$.
\end{itemize}
As the name suggests, this category is idempotent complete, and is the unique idempotent complete category up to equivalence such that $\Setswith{M} \simeq \Setswith{\check{M}}$.
\end{lemma}

\begin{rmk}
\label{rmk:phicheck}
Any semigroup homomorphism $\phi:M \to N$ induces a functor $\check{\phi}: \check{M} \to \check{N}$ mapping $\underline{e}$ to $\underline{\phi(e)}$ and $m: \underline{e} \to \underline{d}$ to $\phi(m) : \underline{\phi(e)} \to \underline{\phi(d)}$, and this in turn induces the essential geometric morphism $\Setswith{M} \to \Setswith{N}$ corresponding to $\phi$ under the $2$-equivalence mentioned at the start of Section \ref{ssec:essential}.
\end{rmk}

Extending Proposition \ref{prop:slice}, we have the following result.

\begin{corollary}
\label{crly:idcomplete}
Suppose that $\Ccal$ is an idempotent-complete category and $X$ is a presheaf on it. Then $\int_{\Ccal} X$ is idempotent complete. It follows that this is (up to equivalence) the unique idempotent complete category with
\begin{equation*}
\Setswith{\Ccal}/X ~\simeq~ \PSh \left(\int_{\Ccal} X\right).
\end{equation*}
\end{corollary}
\begin{proof}
Let $(C,a)$ be an object in $\int_{\Ccal} X$ and suppose that there is an idempotent morphism $e: (C,a) \to (C,a)$ indexed by a morphism $e \in \Ccal$. Then $e$ must itself be an idempotent. Consider the splitting $e = sr$ of $e$ in $\Ccal$; let $D$ be the domain of $s$. Since $a \cdot e = a$, we have $r: (C,a) \to (D,a \cdot r)$ and $s : (D,a \cdot r) \to (C,a)$ in $\int_{\Ccal} X$. This defines the desired splitting of the original idempotent.
\end{proof}

We can also use the idempotent completion to characterize toposes of discrete monoid actions amongst presheaf toposes.

\begin{lemma}
\label{lem:retracted}
Let $\Ccal$ be a category. Then there is a monoid $M$ such that $\PSh(\Ccal) \simeq \PSh(M)$ if and only if there is an object $C$ in $\Ccal$ such that every other object in $\Ccal$ is a retract of $C$. In this case, we can take $M = \End_\Ccal(C)$.
\end{lemma}
\begin{proof}
If $C$ is an object such that every other object is a retract of $C$, then consider the full subcategory of $\Ccal$ on the single object $C$, which we can identify with the monoid $\End_\Ccal(C)$. The idempotent completions of $\End_\Ccal(C)$ and $\Ccal$ agree, so $\PSh(\Ccal) \simeq \PSh(\End_\Ccal(C))$.

Conversely, suppose that $\PSh(\Ccal) \simeq \PSh(M)$ for some monoid $M$. Then there is an object $C'$ in the idempotent completion $\check{\Ccal}$ of $\Ccal$ such that every other object is a retract of $C'$. Because $C'$ lies in the idempotent completion, it is itself a retract of an object $C$ in $\Ccal$. The statement of the lemma then follows from transitivity of retracts.
\end{proof}

\section{The (surjection, inclusion) and (hyperconnected, localic) factorizations}
\label{sec:si-hl}

There are a number of standard factorization systems for geometric morphisms, some applicable to all morphisms, others only to particular classes. While we describe a variety of them here, we focus on the cases which do not take us outside the realm of presheaf toposes, and especially on refinements of these which keep us in the realm of toposes of discrete monoid actions. We begin with factorization systems for essential geometric morphisms, and use these to factorize more general geometric morphisms later.

\begin{dfn}
\label{dfn:gmbasics}
Recall that a geometric morphism $f: \Fcal \to \Ecal$ is:
\begin{itemize}
	\item a \textbf{surjection} if $f^*$ is faithful;
	\item an \textbf{inclusion} if $f_*$ is full and faithful;
	\item \textbf{hyperconnected} if $f^*$ is full and faithful \textit{and} has image closed under subquotients;
	\item \textbf{localic} if every object in $\Fcal$ is a subquotient of an object of the form $f^*(X)$ for some $X$ in $\Ecal$.
\end{itemize}
There are equivalent characterizations of these classes of geometric morphism which we shall employ at various points.
\end{dfn}

The two factorization systems for geometric morphisms that are most well-known are the (surjection, inclusion) factorization, and the (hyperconnected, localic) factorization. That is, every geometric morphism canonically factors as a surjection followed by an inclusion, or as a hyperconnected morphism followed by a localic morphism, uniquely up to compatible equivalence of the intermediate topos. Moreover, these factorizations are compatible in the sense illustrated by \eqref{eq:3part} below.

\subsection{The essential case}

Conveniently, both of these factorizations restrict in a canonical way to the class of essential geometric morphisms between presheaf toposes.

\begin{lemma}
\label{lem:essfacts}
Suppose $f : \Setswith{\Ccal} \to \Setswith{\Dcal}$ is an essential geometric morphism induced by a functor $F : \Ccal \to \Dcal$. Then,
\begin{itemize}
\item $f$ is surjective $\Leftrightarrow$ $F$ is essentially surjective up to retracts;
\item $f$ is an inclusion $\Leftrightarrow$ $F$ is full and faithful;
\item $f$ is hyperconnected $\Leftrightarrow$ $F$ is full and essentially surjective up to retracts;
\item $f$ is localic $\Leftrightarrow$ $F$ is faithful.
\end{itemize}
Here ``essentially surjective up to retracts'' means that for every $D$ in $\Dcal$ there is some $C$ in $\Ccal$ such that $D$ is a retract of $F(C)$. 

In particular, since we can factor any functor $F$ between idempotent-complete small categories as a functor which is essentially surjective up to retracts followed by one which is full and faithful, we obtain a canonical representation of the (surjection, inclusion) factorization of $f$, and the intermediate topos is a presheaf topos. The analogue is true for the (hyperconnected, localic) factorization of $f$.
\end{lemma}
\begin{proof}
See Johnstone, \cite[Examples A4.2.7(b), A4.2.12(b), A4.6.2(c) and A4.6.9]{Ele}; the case of monoid homomorphisms is even explicitly discussed after Example A4.6.9 there. To sketch a short proof, one can verify directly that the given conditions are sufficient; conversely, since the stated factorizations of $F$ must give factorizations of $f$ which we know to be unique up to equivalence of the intermediate topos (or equivalently up to equivalence of the intermediate idempotent-complete category) and since the given conditions are invariant under equivalence, they must also be necessary.
\end{proof}

Let $\phi:M \to N$ be a semigroup homomorphism. Applying Lemma \ref{lem:essfacts} to the functor $\check{\phi}: \check{M} \to \check{N}$ from Remark \ref{rmk:phicheck} and the corresponding essential geometric morphism $f:\Setswith{M} \to \Setswith{N}$, we deduce the following corollaries for essential geometric morphisms between toposes of discrete monoid actions.

\begin{crly}
\label{crly:surj-inj}
The \textbf{(surjection, inclusion) factorization} of $f$ is canonically represented by the factorization of $\phi:M\to N$ as a monoid homomorphism followed by an inclusion of semigroups of the form $\iota : eNe \hookrightarrow N$, where $e = \phi(1)$ is the idempotent of $N$ which is the image of the identity element of $M$.
\begin{equation*}
\begin{tikzcd}[row sep=small,column sep=huge]
M \ar[r,"{\psi}"] & eNe \ar[r,hook,"{\iota}"] & N \\
\Setswith{M} \ar[r,"{\text{\emph{surjection}}}"'] & \Setswith{eNe} \ar[r,"{\text{\emph{inclusion}}}"'] & \Setswith{N}
\end{tikzcd}
\end{equation*}
In particular, essential geometric morphisms induced by monoid homomorphisms are always surjective. Conversely, given an essential surjection, the inclusion part of its (surjection, inclusion) factorization must be an equivalence. That is, the inclusion $\iota : eNe \to N$ of the image of the corresponding semigroup homomorphism induces an equivalence of toposes. We may therefore assume, up to replacing the monoid presenting the codomain topos with a Morita-equivalent one, that an essential surjection is induced by a monoid homomorphism, rather than a mere semigroup homomorphism.
\end{crly}

\begin{crly}
\label{crly:hype-loc}
The \textbf{(hyperconnected, localic) factorization} of $f$ corresponds to the (quotient, injection) factorization of $\phi$, which factors $\phi: M \to N$ through the quotient monoid homomorphism $\pi: M \to M/{\sim}$, where $m \sim n$ if and only if $\phi(m) = \phi(n)$. Diagrammatically:
\begin{equation*}
\begin{tikzcd}[row sep=small,column sep=huge]
M \ar[r,"{\pi}",two heads] & M/{\sim} \ar[r,hook,"{\psi}"] & N \\
\Setswith{M} \ar[r,"{\text{\emph{hyperconnected}}}"'] & \Setswith{M/{\sim}} \ar[r,"{\text{\emph{localic}}}"'] & \Setswith{N}
\end{tikzcd}
\end{equation*}
\end{crly}

These two factorization systems are compatible: we can factorize any semigroup homomorphism $\phi$ and the corresponding essential geometric morphism $f$ into three parts. We factorize $\phi$ into a quotient map $\pi: M \to M/{\sim}$, followed by an injective monoid homomorphism $\psi : M/{\sim} \to eNe$, followed by an inclusion $\iota : eNe \to N$, where $e = \phi(1)$. The induced geometric morphisms give a factorization as follows:
\begin{equation}
\label{eq:3part}
\begin{tikzcd}[row sep=small,column sep=huge]
M \ar[r,"{\pi}",two heads] & M/{\sim} \, \ar[r,hook,"{\psi}"] 
& eNe \ar[r,"{\iota}",hook] & N \\
\Setswith{M} \ar[r,"{\text{hyperconnected}}"'] & \Setswith{M/{\sim}} \ar[r,"{\text{localic surj.}}"'] & \Setswith{eNe} \ar[r,"{\mathrm{inclusion}}"'] & \Setswith{N}
\end{tikzcd}
\end{equation}
In this situation, the surjective part is the composition of the hyperconnected and localic surjection parts, and the localic part is the composition of the localic surjection and inclusion parts. It will often be helpful to consider the three parts of this (hyperconnected, localic surjection, inclusion) factorization separately.

\subsection{The general case}

Unfortunately, the latter part of Lemma \ref{lem:essfacts} is not true for a geometric morphism that is not essential: the intermediate topos in the (surjection, inclusion) or (hyperconnected, localic) factorization of a typical geometric morphism $g: \Setswith{M} \to \Setswith{N}$ is not a presheaf topos, let alone a topos of discrete monoid actions\footnote{On the other hand, it was recently shown by the second author that the (hyperconnected, localic) factorization does have a presentation in terms of actions of \textit{topological} monoids, \cite{TTMA}.}. Nonetheless, we can identify conditions on $\fllr{N}{M}$-sets which produce morphisms in these classes.

\begin{proposition} \label{prop:characterization-localic}
Let $f : \Setswith{M} \to \Setswith{N}$ be the geometric morphism corresponding to the $\fllr{N}{M}$-set $A$. Then $f$ is localic if and only if $M$ is a retract of some subobject of $A$, as a right $M$-set.
\end{proposition}
\begin{proof}
By definition, $f$ being localic requires that every object $X$ of $\Setswith{M}$ be a subquotient of one of the form $f^*(Y)$. By pulling back along (the image under $f^*$ of) the cover of $Y$ by a disjoint union of copies of $N$, we conclude that $f$ is localic if and only if for every object $X$ in $\Setswith{M}$ there is a subobject $C \subseteq \bigsqcup_{i \in I} A$ and an epimorphism $C \to X$. In the special case where $X = M$, we get a surjection $p : C \too M$, which splits since $M$ is projective. Letting $A'$ be the connected component of $C$ containing the section, we conclude that $A'$ must be a subobject of just one of the copies of $A$, and hence by restricting $p$ to $A'$, we conclude that $M$ is a subquotient of $A$, as required. Conversely, if $A$ has a subobject $A'$ of which $M$ is a retract, then each object $X$ in $\Setswith{M}$ admits a surjection $\bigsqcup_{i \in I} A' \too \bigsqcup_{i \in I} M \too X$. Since $\bigsqcup_{i \in I} A'$ is a subobject of $\bigsqcup_{i \in I} A$, this shows that $f$ is localic. 
\end{proof}

We shall extend this proposition to a necessary and sufficient condition for the direct image $f_*$ to be faithful in Scholium \ref{schl:characterization-spread}, but fullness of $f_*$ is challenging in general. We shall at least see a sufficient condition for $f$ to be an inclusion in Corollary \ref{crly:characterization-pure-spread}.

We can characterize surjections in terms of an algebraic condition, albeit a not very enlightening one.

\begin{lemma} \label{lem:characterization-surj}
Let $f : \Setswith{M} \to \Setswith{N}$ be the geometric morphism corresponding to the $\fllr{N}{M}$-set $A$. Then $f$ is a surjection if and only if for all $N$-sets $X$ and elements $x,y \in X$, if we have $x \otimes a = y \otimes a$ in $X \otimes_N A$ for all $a \in A$, then $x = y$.
\end{lemma}
\begin{proof}
Composing with the canonical essential surjective point of $\Setswith{M}$, we see that $f$ is a surjection if and only if the composite point $\Set \to \Setswith{N}$ is. The stated condition is a translation of the requirement that the unit of this point is a monomorphism. The statement then follows from the classical result that the unit of an adjunction is a monomorphism if and only if the left adjoint is faithful.
\end{proof}

Finding necessary and sufficient conditions for $f^*$ to be full is difficult, but fortunately we have other ways to characterize hyperconnected morphisms.

\begin{proposition} \label{prop:characterization-hc}
Let $f : \Setswith{M} \to \Setswith{N}$ be the geometric morphism corresponding to the $\fllr{N}{M}$-set $A$. Then $f$ is hyperconnected if and only if the condition of Lemma \ref{lem:characterization-surj} is satisfied and every sub-$M$-set of $A$ is of the form $I \otimes_N A$ for some right ideal $I \subseteq N$.
\end{proposition}
\begin{proof}
If $f$ is hyperconnected, then $f$ is certainly a surjection. Moreover, since $f^*$ is full and faithful and closed under subobjects, every monomorphism $A' \hookrightarrow A$ must be of the form $f^*(g)$ for some right $N$-set homomorphism $g: X \to N$. But $f^*$ preserves epimorphisms and monomorphisms, which means that if we take the epi-mono factorization of $g$, the epimorphic part must be sent to an isomorphism by $f^*$, so the monomorphic part induces the same subobject. The conclusion follows, since sub-right-$N$-sets of $N$ are precisely right ideals.

Conversely, given the conditions on $A$, we know from Lemma \ref{lem:characterization-surj} that $f$ is a surjection; we shall show that $f^*$ is closed under subobjects. Indeed, given a right $N$-set $X$, consider the image under $f^*$ of a cover of $X$ by copies of $N$, which simplifies to $\coprod_{k \in K} A \too f^*(X)$. Given a subobject $Z$ of $f^*(X)$ in $\Setswith{M}$, we can pull back the cover to obtain a cover of $Z$ of the form $\coprod_{k \in K} I_k \otimes_N A \too Z$ (taking advantage of the fact that subobjects of coproducts are coproducts of subobjects). This lifts to a morphism $\coprod_{k \in K} I_k \hookrightarrow \coprod_{k \in K} N \too X$; applying $f^*$ to the epi-mono factorization of this composite produces the desired presentation of $Z$. We use the characterization of hyperconnected morphisms from \cite[Proposition A4.6.6]{Ele} as those surjections whose inverse image functors are closed under subobjects to conclude that $f$ is hyperconnected.
\end{proof}

\section{The (terminal-connected, \'etale) factorization}
\label{sec:tc-etale}

There is a well-known (connected and locally connected, \'etale) factorization system for locally connected morphisms, constructed for a given locally connected morphism $f$ by slicing the codomain topos over the object $f_!(1)$; see \cite[Lemma C3.3.5]{Ele}. This factorization system extends to essential geometric morphisms, as observed by Caramello in \cite[\S4.7]{caramello-denseness}. Recent work of Osmond \cite[Theorem 5.4.10]{osmond-2geometries} demonstrates how this can be extended to a factorization system for arbitrary geometric morphisms, after replacing \'etale geometric morphisms by more general pro-\'etale geometric morphisms.

We begin from the following definitions, which appear as \cite[Definitions 5.2.3 and 5.3.3]{osmond-2geometries}, respectively.

\begin{dfn}
A geometric morphism $f:\Fcal \to \Ecal$ is said to be \textbf{\'etale} if $\Fcal$ is equivalent to $\Ecal/X$ for some object $X$, and $f$ factors as the equivalence followed by the canonical geometric morphism $\Ecal/X \to \Ecal$; we refer to the latter as the \textit{\'etale geometric morphism corresponding to $X$}.

On the other hand, a geometric morphism $f:\Fcal \to \Ecal$ is said to be \textbf{terminal-connected} if there is a bijection $\Hom_{\Fcal}(1,f^*(X)) \cong \Hom_{\Ecal}(1,X)$, natural in $X$.
\end{dfn}

A geometric morphism into $\Set$ is terminal-connected if and only if it is connected. Indeed, it is clear that connected morphisms are always terminal-connected morphisms; conversely, for $p:\Ecal \to \Set$ the global sections morphism, we have
\[\Hom_{\Ecal}(p^*(X),p^*(Y)) \cong \prod_{x \in X}\Hom_{\Ecal}(1,p^*(Y)) \cong \prod_{x \in X}\Hom_{\Set}(1,Y) \cong \Hom_{\Set}(X,Y),\]
naturally in $X$ and $Y$, whence $p^*$ is full and faithful. Terminal-connectedness is usually defined only for the class of essential geometric morphisms. Indeed, for such morphisms we can recover the definition stated in \cite[\S4.7]{caramello-denseness} by adjointness, as Osmond observes in \cite{osmond-2geometries}.

\begin{lemma}[{\cite[Proposition 5.3.4]{osmond-2geometries}}]
\label{lem:esstc}
An essential geometric morphism $f:\Fcal \to \Ecal$ is terminal-connected if and only if $f_!$ preserves the terminal object.
\end{lemma}

Meanwhile, the following result from \cite{Ele} suggests that we should think of \'etale geometric morphisms over a topos $\Ecal$ as corresponding to discrete internal locales (rather than merely as objects of the topos).

\begin{lemma}[{\cite[Lemma C3.5.4]{Ele}}]
A geometric morphism is \'etale if and only if it is localic and its inverse image functor is logical (so preserves exponential objects and the subobject classifier). In particular, \'etale geometric morphisms are essential (in fact, locally connected).
\end{lemma}

\subsection{The essential case}

Given an essential geometric morphism $f:\Fcal \to \Ecal$, there is a factorization
\[\begin{tikzcd}
\Fcal \ar[r, "g"] & \Ecal/f_!(1) \ar[r, "h"] & \Ecal,
\end{tikzcd}\]
where both factors are essential, $g$ is terminal-connected and $h$ is the local homeomorphism corresponding to $f_!(1)$. This \textbf{(terminal-connected, \'etale) factorization} is again unique up to compatible equivalence of the intermediate topos, so we may refer to $g$ as the \textit{terminal-connected part of $f$} and to $h$ as the \textit{\'etale part of $f$}.

\begin{proposition}
\label{prop:finfib}
Let $f : \PSh(\Ccal) \to \PSh(\Dcal)$ be an essential geometric morphism induced by a functor $F : \Ccal \to \Dcal$. Then $F$ has a factorization $\Ccal \to \Bcal \to \Dcal$ into an final functor followed by a discrete fibration (unique up to equivalence). Further, the induced factorization
\begin{equation*}
\PSh(\Ccal) \to \PSh(\Bcal) \to \PSh(\Dcal)
\end{equation*}
coincides with the (terminal-connected, \'etale) factorization of $f$.
\end{proposition}
\begin{proof}
In \cite{street-walters} it is shown that each functor can be factorized as an initial functor followed by a discrete opfibration. By dualizing we get a factorization of a functor into a final functor followed by a discrete fibration. To show that the induced factorization coincides with the (terminal-connected, \'etale) factorization, we write out the factorization explicitly below.

Using the notations from Subsection \ref{ssec:categories-of-elements}, we may consider the factorization:
\begin{equation}
\label{eq:finfib}
\begin{tikzcd}[row sep = small]
\Ccal \ar[r] & \int_{\Dcal} f_!(1) \ar[r] & \Dcal, \\
C \ar[d, "t"] \ar[r, mapsto] & (F(C), x \cdot Ft) \ar[d, "Ft"] \ar[r, mapsto] & F(C) \ar[d, "Ft"] \\
C' \ar[r, mapsto] & (F(C'),x) \ar[r, mapsto] & F(C')
\end{tikzcd}
\end{equation}
where $\int_{\Dcal} f_!(1)$ the category of elements of the object $f_!(1)$ in $\PSh(\Dcal)$, the right-hand functor is the forgetful functor and $x \in f_!(1)(F(C'))$ corresponds to the morphism $\yon(F(C')) \cong f_!(C') \xrightarrow{f_!(!)} f_!(1)$. We omit the proof of the respective properties and uniqueness, although we refer the reader to the original reference for the dual factorization of a functor into an initial functor followed by a discrete opfibration in \cite{street-walters}; we shall use this dual factorization in the next section. 

To see that this produces a (terminal-connected, \'etale) factorization of $f$, we employ Proposition \ref{prop:slice} to observe that the right-hand factor is \'{e}tale, since the extra left adjoint of the geometric morphism induced by the projection functor is identified under the equivalence with the forgetful functor from the slice topos. Meanwhile, the geometric morphism induced by the left-hand factor has left adjoint sending a presheaf $X$ on $\Ccal$ to the object $f_!(X \to 1)$ of $\PSh(\Dcal)/f_!(1)$, whence it is terminal-connected by Lemma \ref{lem:esstc}.
\end{proof}

Now suppose $f: \Setswith{M} \to \Setswith{N}$ is an essential geometric morphism induced by a semigroup homomorphism $\phi : M \to N$, and let $e:= \phi(1)$. In applying Proposition \ref{prop:finfib}, it will be convenient to substitute $\phi$ for the induced functor $\phi:M \to \check{N}$. Then $\Setswith{N}/f_!(1)$ is the category of presheaves on the category of elements $\int_{\check{N}} f_!(1)$, and the terminal-connected part of $f$ is the geometric morphism induced by the functor $M \to \int_{\check{N}} f_!(1)$ such that:
\begin{itemize}
\item the unique object of $M$ is sent to the object $(\underline{e},* \otimes e)$ of $\int_{\check{N}} f_!(1)$, where $* \otimes e \in 1 \otimes_M eN \cong f_!(1)$, and
\item the morphism corresponding to $m \in M$ is sent to the endomorphism indexed by $\phi(m)$ on $X$.
\end{itemize}
Denote the monoid of endomorphisms of $(\underline{e},* \otimes e)$ in $\int_{\check{N}} f_!(1)$ by $D$. More explicitly, since morphisms in this category are indexed by morphisms in $\check{N}$, we can identify $D$ with the following subsemigroup of $N$:
\begin{equation}
\label{eq:D}
D = \{ n \in N : ene = n ~\text{ and }~  * \otimes en = * \otimes e ~\text{ in }~ 1 \otimes_M eN \} \subseteq eNe.
\end{equation}
It will also be useful for us to consider the object $(\underline{1},1\otimes e)$ of $\int_{\check{N}} f_!(1)$; letting $E$ be the monoid of endomorphisms of $(\underline{1},1 \otimes e)$, we can identify $E$ with a submonoid of $N$:
\begin{equation}
\label{eq:E}
E = \{ n \in N : * \otimes en = * \otimes e \text{ in } 1 \otimes_M eN \}.
\end{equation}
In particular, $D = eEe$. Further, we have a diagram of semigroup homomorphisms,
\begin{equation}
\label{eq:5maps}
\begin{tikzcd}
M \ar[r, "\psi"] & D \ar[r] \ar[d, "\iota"'] & eNe \ar[d] \\
                 & E \ar[r, "\tau"']         &  N
\end{tikzcd},
\end{equation}
where the horizontal maps are monoid homomorphisms, the vertical maps are inclusions of subsemigroups which reduce to identities when $\phi$ is a monoid homomorphism, and both paths $M \to N$ compose to give $\phi$.

We will see later that the terminal-connected part of $f$ factors through both $\Setswith{D}$ and $\Setswith{E}$, in such a way that each of the factors is again terminal-connected. First, we give a different interpretation of $D$ and $E$ in terms of right-factorability.

\begin{dfn}
\label{dfn:rfactor}
Recall from \cite[Definition 2.9]{hr1} that a subset $S$ of a monoid $M$ is called \textbf{right-factorable} if whenever $x \in S$ and $y \in M$ with $xy \in S$, then $y \in S$; such a subset automatically contains the identity. Further, for an arbitrary subset $T$, we defined $\langle{T}\rangle\rangle_M \subseteq M$ to be the smallest right-factorable submonoid of $M$ containing $T$. We say that $\langle{T}\rangle\rangle_M$ is the submonoid of $M$ \textbf{right-factorably generated} by $T$.
\end{dfn}

\begin{lemma}
\label{lem:(M))}
In $1 \otimes_M eN$, we have $\ast \otimes en = \ast \otimes e$ if and only if $n \in \langle{\phi(M)}\rangle\rangle_N$. 
With the notation established above, we find that $E = \langle{\phi(M)}\rangle\rangle_{N}$, and similarly that $D = \langle{\phi(M)}\rangle\rangle_{eNe}$.
\end{lemma}
\begin{proof}
By definition of equality in $1 \otimes_M eN$, for $n \in N$ we have $* \otimes e = * \otimes en$ if and only if $1 \sim_M n$, where ${\sim_M}$ is the right congruence generated by the basic relations $1 \sim \phi(m)$ for all $m \in M$. But by \cite[Lemma 2.12]{hr1}, we have $\langle{\phi(M)}\rangle\rangle_{N} = \{ n \in N : 1 \sim_M n \}$. In other words, $E = \langle{\phi(M)}\rangle\rangle_N$. To show the analogous result for $D$, note that $eNe$ is a retract of $eN$ (as left $M$-sets). The functor $1 \otimes_M -$ preserves retracts, in particular $1 \otimes_M eNe \subseteq 1 \otimes_M eN$. So for $n \in eNe$ the equation $\ast \otimes en = \ast \otimes e$ holds in $1 \otimes_M eN$ if and only if it holds in $1 \otimes_M eNe$. The proof that $D = \langle{\phi(M)}\rangle\rangle_{eNe}$ is now analogous to the above proof that $E = \langle{\phi(M)}\rangle\rangle_N$. 
\end{proof}

So the diagram \eqref{eq:5maps} can be written more explicitly as
\begin{equation}
\label{eq:5maps2}
\begin{tikzcd}
M \ar[r, "\psi"] & \langle{\phi(M)}\rangle\rangle_{eNe} \ar[r] \ar[d, "\iota"'] & eNe \ar[d] \\
                 & \langle{\phi(M)}\rangle\rangle_N \ar[r, "\tau"']         &  N
\end{tikzcd}.
\end{equation}

It turns out that $f$ is terminal-connected if and only if the inclusion $\tau$ in \eqref{eq:5maps2} is the identity.

\begin{corollary}
\label{cor:tc}
An essential geometric morphism $f : \Setswith{M} \to \Setswith{N}$ induced by a semigroup homomorphism $\phi:M \to N$ is terminal-connected if and only if $\langle{\phi(M)}\rangle\rangle_{N} = N$.
\end{corollary}
\begin{proof}
Let $e = \phi(1)$. Then $f_!(1) = 1 \otimes_M eN$, whence $f$ is terminal-connected if and only if $1 \otimes_M eN \simeq 1$, which is to say that $* \otimes e = * \otimes en$ for all $n \in N$. By Lemma \ref{lem:(M))}, this is equivalent to $\langle{\phi(M)}\rangle\rangle_N = N$.
\end{proof}

\begin{example} \ 
\label{xmpl:tc}
\begin{enumerate}
\item Let $\Zbb$ be the group of integers under addition and let $\Nbb \subseteq \Zbb$ be the submonoid of natural numbers. For each $n \in \Nbb$, we have that $n + (-n) = 0$, and as a result $-n$ is contained in the right-factorable submonoid generated by $\Nbb$, and hence $\langle{\Nbb}\rangle\rangle_\Zbb = \Zbb$. It follows that the induced essential localic surjection $\Setswith{\Nbb} \to \Setswith{\Zbb}$ is terminal-connected.
\item More generally, let $G$ be a group and let $N \subseteq G$ be a submonoid such that $G = \{ a^{-1}b ~:~ a,b \in N \}$; for example, if $R$ is a valuation ring in a field $K$ then we may take $N = R-\{0\}$ and $G = K^*$ to be the respective multiplicative monoids of non-zero elements. Then for every $a,b \in N$ we have $a(a^{-1}b) = b$, and this shows that $\langle{N}\rangle\rangle_G = G$. So the induced geometric morphism $\Setswith{N} \to \Setswith{G}$ is terminal-connected.
\item Consider any ring $R$ as a monoid with its multiplication operation, and consider the subsemigroup $\{0\} \subseteq R$. Because $0 \cdot r = 0$ for all $r \in R$, we have that the right-factorable submonoid generated by $\{0\}$ is equal to $R$ itself. It follows that the induced essential geometric morphism $\Set \simeq \Setswith{\{0\}} \to \Setswith{R}$ is terminal-connected. 
\item Let $M$ be a commutative idempotent monoid. We denote the multiplication in $M$ by $\wedge$, and in this way we can view $M$ as a meet-semilattice. Let $N \subseteq M$ be a subsemigroup, i.e.\ a subset closed under $\wedge$. We can compute that $\langle{N}\rangle\rangle_M$ is then the upwards closure of $N$. So $\Setswith{N} \to \Setswith{M}$ is terminal-connected if and only if the upwards closure of $N \subseteq M$ is $M$.
\end{enumerate}
\end{example}

\begin{prop}
\label{prop:tc}
Given a semigroup homomorphism $\phi:M \to N$, let 
\begin{equation*}
f : \Setswith{M} \to \Setswith{N}
\end{equation*}
be the induced essential geometric morphism. Then the terminal-connected part of $f$ has (surjection, inclusion) factorization given by
\begin{equation}
\label{eq:tctc}
\begin{tikzcd}
\Setswith{M} \ar[r,"{k}"] & \Setswith{\langle{\phi(M)}\rangle\rangle_{eNe}} \ar[r,"{j}"] & \Setswith{N}/f_!(1).
\end{tikzcd}
\end{equation}
where $k$ is the essential surjection induced by the factor
\[\psi : M \to \langle{\phi(M)}\rangle\rangle_{eNe}\]
of $\phi$ from \eqref{eq:5maps2}, and $j$ is the essential inclusion induced by the inclusion of the monoid $\langle{\phi(M)}\rangle\rangle_{eNe}$ as a full subcategory of $\int_{\check{N}} f_!(1)$ on the single object $(\underline{e},* \otimes e)$. Moreover, both $k$ and $j$ are terminal-connected.
\end{prop}
\begin{proof}
That this is a canonical representation of the (surjection, inclusion) factorization follows from Lemma \ref{lem:essfacts}, so we only need to verify the last claim. We can deduce from Lemma \ref{lem:(M))} that $D$ can be identified with $\langle{\psi(M)}\rangle\rangle_{eNe}$, whence $k$ is terminal-connected by Corollary \ref{cor:tc}. To see that $j$ is terminal-connected, observe that $j_!(1) \cong j_!k_!(1) \cong 1$ since both $k$ and $j \circ k$ are terminal-connected.
\end{proof}

\begin{prop}
\label{prop:tc2}
With the same set-up as Proposition \ref{prop:tc}, the geometric inclusion $j$ in \eqref{eq:tctc} further factors as,
\begin{equation}
\label{eq:tctc2}
\begin{tikzcd}
\Setswith{\langle{\phi(M)}\rangle\rangle_{eNe}} \ar[r,"{j_1}"] & \Setswith{\langle{\phi(M)}\rangle\rangle_N} \ar[r,"{j_2}"] & \Setswith{N}/f_!(1).
\end{tikzcd}
\end{equation}
where $j_1$ is the essential inclusion induced by the inclusion of semigroups 
\begin{equation*}
\iota:\langle{\phi(M)}\rangle\rangle_{eNe} \to \langle{\phi(M)}\rangle\rangle_N 
\end{equation*}
from \eqref{eq:5maps2}, and $j_2$ is the essential inclusion induced by the inclusion of $\langle{\phi(M)}\rangle\rangle_N$ as a full subcategory of $\int_{\check{N}} f_!(1)$ on the object $(\underline{1},1 \otimes e)$. Again, both $j_1$ and $j_2$ are terminal-connected.
\end{prop}
\begin{proof}
Replacing all of the monoids with their idempotent completions and extending semigroup homomorphisms to functors accordingly, the fact that these geometric morphisms are inclusions is another application of Lemma \ref{lem:essfacts}. Because, $\langle{\phi(M)}\rangle\rangle_{eNe}$ contains $\phi(M)$, we have that $\langle{\phi(M)}\rangle\rangle_{eNe}$ right-factorably generates $\langle{\phi(M)}\rangle\rangle_N$, so using Corollary \ref{cor:tc} we see that $j_1$ is terminal-connected. That $j_2$ is terminal-connected follows just as for $j$ in the proof of Proposition \ref{prop:tc}.
\end{proof}

In summary, the geometric morphism $f : \Setswith{M} \to \Setswith{N}$ induced by a semigroup morphism $\phi : M \to N$, with $\phi(1) = e$, factors as
\begin{equation*}
\begin{tikzcd}
\Setswith{M} \ar[r,"{\text{tc surj.}}"] & \Setswith{\langle{\phi(M)}\rangle\rangle_{eNe}} \ar[d,"{\text{tc incl.}}"] & & \\
& \Setswith{\langle{\phi(M)}\rangle\rangle_N} \ar[r,"{\text{tc incl.}}"] & \Setswith{N}/f_!(1) \ar[d,"{\text{\'etale}}"] & \\
& & \Setswith{N},
\end{tikzcd}
\end{equation*}
where `tc' is short-hand for terminal-connected. If $\phi(1)=1$, then this reduces to
\begin{equation} \label{eq:tc-etale-monoid-morphism}
\begin{tikzcd}
\Setswith{M} \ar[r,"{\text{tc surj.}}"] & \Setswith{\langle{\phi(M)}\rangle\rangle_N} \ar[r,"{\text{tc incl.}}"] & \Setswith{N}/f_!(1) \ar[r,"{\text{\'etale}}"] & \Setswith{N}.
\end{tikzcd}
\end{equation}

Note that since $f$ is induced by a semigroup homomorphism, $f_!(1) = 1 \otimes_M \phi(1)N$ is an inhabited set, so the map $f_!(1) \to 1$ is an epimorphism, which implies that the induced geometric morphism $\Setswith{N}/f_!(1) \to \Setswith{N}$ is always a surjection. Indeed, since $\PSh(N)$ is hyperconnected, all of its non-initial objects are well-supported, so an \'{e}tale morphism from any non-degenerate topos will be a surjection.

We now consider conditions under which the morphism $f : \Setswith{M} \to \Setswith{N}$ we have been considering is \'etale. By orthogonality, if $f$ is \'etale, then the terminal-connected factor $\Setswith{M} \to \Setswith{N}/f_!(1)$ must be an equivalence, and by the above remark, $f$ must be a surjection. Since the (surjection, inclusion) factorization of $f$ is given by $\Setswith{M} \to \Setswith{eNe} \to \Setswith{N}$, with $e = \phi(1)$, it follows that $f$ is \'etale if and only if the inclusion part $\Setswith{eNe} \to \Setswith{N}$ is an equivalence and the surjection part $\Setswith{M} \to \Setswith{eNe}$ is \'etale. So we may without loss of generality restrict our attention to the case where $\phi$ is a monoid homomorphism.

\begin{proposition}
Let $f: \Setswith{M} \to \Setswith{N}$ be a geometric morphism induced by a monoid homomorphism $\phi : M \to N$. Then the terminal-connected surjection part in the factorization \eqref{eq:tc-etale-monoid-morphism} is an equivalence if and only if $\phi$ is injective and $\phi(M)$ is a right-factorable submonoid of $N$.
\end{proposition}
\begin{proof}
The terminal-connected surjection part in \eqref{eq:tc-etale-monoid-morphism} is induced by the monoid homomorphism $M \to \langle{\phi(M)}\rangle\rangle_{eNe}$. So we get an equivalence if and only if this monoid homomorphism is a bijection. 
\end{proof}

To understand when the terminal-connected inclusion part of $f$ is an equivalence, we need the following definition.

\begin{dfn}
\label{dfn:ltimes}
Given a monoid $N$, we write $N^{\ltimes}$ for the submonoid of \textbf{right-invertible elements}. That is,
\[N^{\ltimes} := \{u \in N : \exists v \in N, \, uv = 1\}.\]
Dually, we write $N^{\rtimes}$ for the submonoid of \textbf{left-invertible elements}, so
\[N^{\rtimes} := \{v \in N \mid \exists u \in N, \, uv = 1\}.\]
\end{dfn}

\begin{proposition}
\label{prop:tcincequiv}
Let $f: \Setswith{M} \to \Setswith{N}$ be a geometric morphism induced by a monoid homomorphism $\phi : M \to N$. Then the terminal-connected inclusion part in the factorization \eqref{eq:tc-etale-monoid-morphism} is an equivalence if and only if for all $n \in N$ there is some $u \in N^\ltimes$ such that $nu \in \langle{\phi(M)}\rangle\rangle_N$. 
\end{proposition}
\begin{proof}
The terminal-connected inclusion part $\Setswith{\langle{M}\rangle\rangle_N} \to \Setswith{N}/f_!(1)$ is an essential inclusion induced by the functor $\langle{M}\rangle\rangle_N \to \int_N f_!(1)$. It is enough to show that this geometric morphism is surjective as well, which by Lemma \ref{lem:essfacts} is the case if and only if the functor $\langle{M}\rangle\rangle_N \to \int_N f_!(1)$ is essentially surjective up to retracts. In other words, we need that every element $\ast \otimes n \in f_!(1)$ is a retract of $\ast \otimes 1$ in the category of elements $\int_N f_!(1)$. Equivalently, for each $n \in N$ there are $u,v \in N$ such that $uv = 1$ and $\ast \otimes nu = \ast \otimes 1$. By the proof of Lemma \ref{lem:(M))}, $\ast \otimes nu = \ast \otimes 1$ if and only if $nu \in \langle{M}\rangle\rangle_N$. 
\end{proof}

Combining the two propositions above, we obtain a characterization of \'etale geometric morphisms induced by monoid homomorphisms.
\begin{theorem}
\label{thm:etale}
Let $f$ be an essential geometric morphism induced by a monoid homomorphism $\phi : M \to N$. Then the following are equivalent:
\begin{enumerate}
\item $f$ is \'etale;
\item $\phi$ is injective, $\phi(M) \subseteq N$ is right-factorable and for any $n \in N$ there is some $u \in N^{\ltimes}$ such that $nu \in \phi(M)$.
\end{enumerate}
More generally, if $\phi$ is merely a semigroup homomorphism, then $f$ is \'etale if and only if the monoid homomorphism part of $\phi$ satisfies the conditions above, and the inclusion $eNe \subseteq N$ induces an equivalence, where $e=\phi(1)$.
\end{theorem} 

Further, we remark that \'etale geometric morphisms are always essential, so all \'etale geometric morphisms $\Setswith{M} \to \Setswith{N}$ are induced by a semigroup homomorphisms.

\begin{corollary}
Let $f: \Setswith{M} \to \Setswith{N}$ be an \'etale geometric morphism. If $N^{\ltimes} = \{1\}$, then $f$ is an equivalence.
\end{corollary}
\begin{proof}
Since $f$ is \'etale, it is essential. So up to equivalence, it is induced by some semigroup homomorphism $\phi : M \to N$. If $\phi(1)=e$, then the inclusion $eMe \subseteq M$ induces an equivalence. Because $N^\ltimes=\{1\} $, this implies $e=1$, see \cite[Corollary 6.2(3)]{knauer-morita}. So $\phi$ is a monoid homomorphism. Applying Theorem \ref{thm:etale}, for all $n \in N$, there is $u \in N^{\ltimes}$ such that $nu \in \phi(M)$. Because $N^{\ltimes}=\{1\}$, this means that $\phi$ is bijective, so $f$ is an equivalence.
\end{proof}

\begin{example} \label{eg:etale-geometric-morphisms} \ 
\begin{enumerate}
\item For $H \subseteq G$ an inclusion of groups, we have that the induced geometric morphism $\Setswith{H} \to \Setswith{G}$ is \'etale.
\item Consider the monoid $\Zbb_p^\ns$ of nonzero $p$-adic integers under multiplication. Each $p$-adic integer can be written as $up^k$ for $u \in \Zbb_p$ a unit $k \in \{0,1,2,\dots\}$. Further, if $x$ is a nonzero $p$-adic integer, then $xp^k = p^l$ implies $x = p^{l-k}$. From this it follows that the inclusion $\mathbb{N} \to \Zbb_p^\ns,~ k \mapsto p^k$ induces an \'etale geometric morphism $\Setswith{\Nbb} \to \Setswith{\Zbb_p^\ns}$, where $\Nbb$ is the monoid of natural numbers under addition.
\end{enumerate}
\end{example}

\subsection{The general case}

Since all \'{e}tale morphisms between toposes of monoid actions are essential, there is limited benefit to considering \'{e}tale geometric morphisms induced by $\fllr{M}{N}$-sets. However, we may still consider more general terminal-connected morphisms.

By definition, a geometric morphism $f:\PSh(M) \to \PSh(N)$ induced by an $[N,M)$-set $A$ is terminal-connected if and only if $\Hom_N(1,X) \cong \Hom_M(1,X \otimes_N A))$, naturally in $X$. Translating this into algebra, this is equivalent to requiring that for every right $N$-set $X$, the $M$-fixed points of $X \otimes_N A$ are all of the form $x \otimes a$, for $x$ an $N$-fixed point of $X$. Indeed, consider the mapping
\begin{align*}
\Hom_N(1,X) &\to \Hom_M(1,X \otimes_N A)) \\
x &\mapsto x \otimes a.
\end{align*}
This is independent of the choice of $a \in A$, since $x \otimes a = x \otimes (n \cdot a)$ for all $n \in N$ and $A$ is connected as a left $N$-set since it is filtered; $x \otimes a$ is an $M$-fixed point by a similar argument. Also, if $x,y$ are distinct $N$-fixed points of $X$, then we have a monomorphism $(x,y):1+1 \hookrightarrow X$ which is preserved by $A$ due to flatness, whence we have a monomomorphism $(x \otimes a,y \otimes a) \hookrightarrow X \otimes_N A$ and the mapping above is injective. So terminal-connectedness reduces to the requirement of surjectivity of this mapping.

Now observe that, while $\Hom_N(1,-)$ does not preserve arbitrary colimits, it does preserve the expression of an $N$-set as a colimit of its principal sub-$N$-sets. As such, we can reduce the condition to the special case of principal $N$-sets, to conclude:
\begin{lemma}
\label{lem:tccond}
A geometric morphism $f:\PSh(M) \to \PSh(N)$ induced by an $[N,M)$-set $A$ is terminal-connected if and only if for every principal $N$-set $X$, the $M$-fixed points of $X \otimes_N A$ are all of the form $x \otimes a$, for $x$ an $N$-fixed point of $X$.
\end{lemma}

\begin{example}
Consider the geometric morphism $f : \PSh(\Zbb) \to \PSh(\Nbb)$ given by the $\fllr{\Nbb}{\Zbb}$-set $\Zbb$, with $\Nbb$ and $\Zbb$ both seen as monoids under addition. For integers $a\geq 0$ and $b \geq 1$, we write $N_{a,b}$ for the quotient of the right $\Nbb$-set $\Nbb$ by the congruence generated by $a \sim a+b$. The elements of $N_{a,b}$ can be written as $\{0,1,\dots,a+b-1\}$ and the generator $1 \in \Nbb$ acts by sending each $x$ to $x+1$ for $x \leq a +b-2$ and by sending $a+b-1$ to $a$. Every principal $N$-set is either isomorphic to $\Nbb$ or to some $N_{a,b}$. We compute $\Nbb\otimes_\Nbb \Zbb \cong \Zbb$ and $N_{a,b} \otimes_\Nbb \Zbb \cong \Zbb/b\Zbb$. So for a principal right $\Nbb$-set $X$, either both $X$ and $X \otimes_{\Nbb} \Zbb$ have no fixed points, or they both have precisely one fixed point. It follows that $f$ is terminal-connected. 
\end{example}

\begin{rmk}
Recently, Osmond in \cite{osmond-2geometries} identified the correct generalization of the (terminal-connected, \'{e}tale) factorization to be the (terminal-connected, \mbox{pro-\'{e}tale}) factorization. From \cite[Definition 5.4.6]{osmond-2geometries}, a geometric morphism $f:\Fcal \to \Ecal$ is \textbf{pro-\'etale} if it can be expressed as a cofiltered bilimit of \'etale morphisms over $\Ecal$. It might therefore be interesting to investigate this factorization system and its application to the special case of geometric morphisms between toposes of discrete monoid actions, but we leave this to future work.
\end{rmk}

\section{The (pure, complete spread) factorization}
\label{sec:pure-complete-spread}

In this section, we discuss the (pure, complete spread) factorization, sometimes called the comprehensive factorization. We shall see that this factorization system is dual in a concrete sense to the (terminal-connected, \'etale) factorization system of the last section; see Proposition \ref{prop:iniopf}, for example.

We first recall the definition of pure geometric morphisms.

\begin{dfn} \label{def:pure}
A geometric morphism $f:\Fcal \to \Ecal$ is said to be \textbf{dominant} if the canonical map $0 \to f_*(0)$ is an isomorphism.

A geometric morphism $f:\Fcal \to \Ecal$ is \textbf{pure} if the natural map $1 \sqcup 1 \to f_*(1 \sqcup 1)$ is an isomorphism. Equivalently, $f$ is pure if and only if $f_*$ preserves finite coproducts \cite[C3.4.12(i)]{Ele}, so in particular any pure geometric morphism is dominant.
\end{dfn}

\begin{remark}
Bunge and Funk's definition of pure geometric morphism in \cite{bunge-funk-spreads-I} and \cite{bunge-funk-spreads-II} is different (aside from the fact that we take as fixed base topos $\Scal = \Set$): there a geometric morphism is called pure if the the natural map $1 \sqcup 1 \to f_*(1 \sqcup 1)$ is an epimorphism, and pure dense if it is an isomorphism. In Definition \ref{def:pure}, we follow the convention originally used by Johnstone \cite{johnstone-factorization-II}, which was later also followed by Bunge and Funk in \cite{bunge-funk}. We recommend all of these references to a reader interested in a deeper treatment of these properties.
\end{remark}

\begin{lemma}
All geometric morphisms $f:\Setswith{M} \to \Setswith{N}$ are dominant.
\end{lemma}
\begin{proof}
Suppose $f$ is induced by the $\fllr{N}{M}$-set $A$. Then $f_*(0) = \Hom_M(A,0) = 0$ since $A$ (being flat as a left-$N$-set) is non-empty.
\end{proof}

The following lemma is a special case of \cite[Proposition 2.7]{bunge-funk-spreads-I}, in the setting of locally connected Grothendieck toposes over the base topos $\Scal = \Set$. We give a simplified proof in this special case.

\begin{lemma} \label{lem:characterization-pure-general}
Suppose $f:\Fcal \to \Ecal$ is a geometric morphism (between Grothendieck toposes) and both $\Fcal$ and $\Ecal$ are locally connected. Then the following are equivalent:
\begin{enumerate}
	\item $f$ is a pure geometric morphism;
	\item $f^*$ preserves connected objects;
	\item The unit of $f$ is an isomorphism at objects of the form $p^*(S)$, where $p$ is the global sections geometric morphism of $\Ecal$.
	\item $f_*$ preserves small coproducts.
\end{enumerate}
\end{lemma}
\begin{proof}
Let $q,p$ be the global sections morphisms of $\Fcal,\Ecal$ respectively. 

$(1) \Leftrightarrow (2)$ The object $1 \sqcup 1$ classifies complemented subobjects. That is, for $X$ in $\Ecal$, the complemented subobjects of $X$ correspond to maps $X \to 1 \sqcup 1$. Further, the map $1 \sqcup 1 \to f_*(1 \sqcup 1)$ induces a map $\beta_X : \Hom_\Ecal(X,1 \sqcup 1) \to \Hom_\Ecal(X,f_*(1\sqcup 1)) \simeq \Hom_\Fcal(f^*(X),1\sqcup 1)$. It follows that $f$ is pure if and only if the map $\beta_X$ is a bijection for every $X$ in a generating family $S$ for $\Ecal$. This is the case if and only if for any $X \in S$, the inverse image functor $f^*$ induces a bijection between complemented subobjects of $X$ and complemented subobjects of $f^*(X)$. Because $\Ecal$ is locally connected, this shows that $f$ is pure if and only if $f^*$ preserves connected objects.

$(2) \Rightarrow (3)$ Preservation of connected objects means that $q_!f^* \cong p_!$, since all of these functors preserve the expression of objects of $\Ecal$ as coproducts of connected objects. In particular, their adjoints are isomorphic, which is to say that $p^* \cong f_*q^* = f_*f^*p^*$ (and more specifically this forces the desired unit morphisms to be an isomorphisms).

$(3) \Leftrightarrow (4)$ Given a set $S$ indexing a family of objects $\{X_s \mid s \in S\}$ of objects of $\Fcal$, observe that their coproduct is determined by the collection of pullbacks, 
\[\begin{tikzcd}
X_s \ar[r] \ar[d] \ar[dr, phantom, very near start, "\lrcorner"] & 1 \ar[d, "s"] \\
\coprod_{s \in S}X_s \ar[r] & q^*(S);
\end{tikzcd}\]
this is because coproducts are stable under pullback in a topos. Applying $f_*$, which preserves these pullbacks, we see that we have:
\[\begin{tikzcd}
f_*(X_s) \ar[r] \ar[d] \ar[dr, phantom, very near start, "\lrcorner"] & 1 \ar[d, "s"] \\
f_*(\coprod_{s \in S}X_s) \ar[r] & f_*q^*(S),
\end{tikzcd}\]
and so given that $f_*q^*(S) \cong p^*(S)$ (compatibly with the inclusions of the elements $1 \to p^*(S)$) this expresses $f_*(\coprod_{s \in S}X_s)$ as the coproduct of the objects $f_*(X_s)$, as required. Conversely, if $f_*$ preserves small coproducts, the unit at $p^*(S)$ can be decomposed as $p^*(S) \cong \coprod_{s \in S} f_*(1) \cong f_*(\coprod_{s \in S} 1) = f_*q^*(S)$, so it is an isomorphism.

$(4) \Rightarrow (1)$ Immediate from the definition of pure morphisms.
\end{proof}

It follows from the proof above that to check that $f$ is pure, it is enough to check that $f^*(X)$ is connected, for each $X$ in a generating family for $\Ecal$. 

Connected geometric morphisms are pure, cf. \cite[Lemma C3.4.14]{Ele}. Moreover, a locally connected geometric morphism to $\Set$ is pure if and only if it is connected (if and only if the terminal object is connected). 

\begin{example}
If $X$ is a topological space, then the unique geometric morphism $\Sh(X) \to \Set$ is pure if and only if it is connected, which is the case if and only if $X$ is connected as a topological space. More generally, let $X$ and $Y$ be locally connected topological spaces and let $\phi : Y \to X$ be a continuous map. Then by Lemma \ref{lem:characterization-pure-general} the induced geometric morphism $f : \Sh(Y) \to \Sh(X)$ is pure if and only if the inverse image of any connected open set is connected. In particular, consider an inclusion $\{x\} \subseteq X$ and let $p : \Set \to \Sh(X)$ be the induced geometric morphism. If $X$ is locally connected, then $p$ is pure if and only if $x$ is contained in any connected open subset $U \subseteq X$. This is the case if and only if $x$ is contained in any open subset, i.e.\ if and only if $x$ is a dense point. 

For example, consider the spectrum $\mathrm{Spec}(\Zbb)$ with the Zariski topology (this is a locally connected topological space, because every nonempty open subset is connected). Let $x$ be the generic point of $\mathrm{Spec}(\Zbb)$, corresponding to the prime ideal $(0)$. Then the inclusion $\{x\} \subseteq \mathrm{Spec}(\Zbb)$ induces a pure geometric morphism. Note that the geometric morphism induced by $\{x\} \subseteq \mathrm{Spec}(\Zbb)$ is not surjective, in particular it is not connected.
\end{example}

We now recall the definition of spreads, based on \cite[Definition 1.1]{bunge-funk-spreads-I} and the equivalent conditions of their Proposition 1.5. Note that \cite[Proposition 1.5]{bunge-funk-spreads-I} mentions \textit{definable subobjects}, which are the suitable generalization of complemented subobjects to an arbitrary base topos. Since we are working with Grothendieck toposes over $\Set$ in this paper, the definable subobjects are precisely the complemented subobjects, which simplifies the definition. 

\begin{definition} \label{def:spread}
Let $f : \Fcal \to \Ecal$ be a geometric morphism, and let $S$ be a generating family of $\Ecal$. Then $f$ is a \textbf{spread} if for every object $F$ in $\Fcal$, there is a complemented subobject $C \subseteq \bigsqcup_{i \in I} f^*(X_i)$, with each $X_i \in S$, such that there is an epimorphism $C \to F$. This definition does not depend on the choice of generating family $S$.
\end{definition}

To see that the definition of spread does not depend on the choice of generating family, suppose we have an inclusion $S' \subseteq S$ of generating families for $\Ecal$. Let $C \subseteq \bigsqcup_{i\in I} f^*(X_i)$ be a complemented subobject, with each $X_i \in S$. Then we can find for each $i \in I$ an epimorphism $\bigsqcup_{j \in I(i)} X_{ij} \to X_i$ for some set $I(i)$ and $X_{ij} \in S'$ for all $j \in I(i)$. Pulling back $C$ along the epimorphism 
\begin{equation*}
\bigsqcup_{\substack{i \in I \\ j \in I(i)}} f^*(X_{ij}) \to \bigsqcup_{i \in I} f^*(X_i)
\end{equation*}
gives a complemented subobject $C' \subseteq \bigsqcup_{i \in I,~j\in I(i)} f^*(X_{ij})$, together with an epimorphism $C' \to C$. Any epimorphism $C \to F$ will then extend to an epimorphism $C' \to F$. Thus, given any pair of generating families, we can take their union and conclude via two applications of this argument that the definitions involving the respective families agree.

\begin{proposition}
\label{prop:spreadchar}
Let $f : \Fcal \to \Ecal$ be a geometric morphism. Then $f$ is a spread if and only if every object $F$ in $\Fcal$ there is a complemented subobject $C \subseteq f^*(X)$, for some $X$ in $\Ecal$, and an epimorphism $C \to F$. In particular, every spread is localic.
\end{proposition}
\begin{proof}
The assumption that the generating family is small was not necessary in the discussion above. In particular, taking $S$ the class of all objects in $\Ecal$, we obtain the hypothesised characterization of spreads. This result appears in \cite[Corollary 3.1.8]{bunge-funk}.
\end{proof}

\begin{example}
Let $X$ be a topological space. Then the unique geometric morphism $\Sh(X) \to \Set$ is a spread if and only if $X$ has a basis of clopen subsets (i.e.\ if and only if $X$ is zero-dimensional).
\end{example}

As a special case of spreads, we also include the following definition:
\begin{dfn}
We say a geometric morphism $f:\Fcal \to \Ecal$ is an \textbf{injection} if its direct image $f_*$ is faithful. Equivalently, $f$ is an injection if and only if every object of $\Fcal$ is a quotient of one in the image of $f^*$, whence by Proposition \ref{prop:spreadchar} every injection is a spread.
\end{dfn}

We can prove the equivalence of the two definitions as follows. A general fact from category theory is that $f_*$ is faithful if and only if the counit of the adjunction $f^*f_*X \to X$ is an epimorphism. So if $f_*$ is faithful, then $X$ is a quotient of $f^*f_*X$. Conversely, if $X$ is a quotient of $f^*Y$ for some $Y$ in $\Ecal$, then the quotient map $f^*Y \to X$ factorizes through the counit $f^*f_*X \to X$, so the counit is an epimorphism.

In particular, inclusions are injections, so \textit{a fortiori} inclusions are spreads.

We now discuss the notion of completeness for geometric morphisms with locally connected domain. Again, the definition we give below is simplified by the fact we are working with Grothendieck toposes over $\Set$. For the more general definition, we refer to \cite{bunge-funk}.

Let $f : \Fcal \to \Ecal$ be a geometric morphism with $\Fcal$ locally connected, and let $(\Ccal,J)$ be a site for $\Ecal$. Consider the category $\Conn(f)$ with as objects the pairs $(C,c)$, with $C$ in $\Ccal$ and $c \hookrightarrow f^*(C)$ a connected component, and as morphisms $(C,c) \to (C',c')$ the maps $a : C \to C'$ such that the image of $c$ along $f^*(a)$ is contained in $c'$. There is a flat functor $\Conn(f) \to \Fcal$ that sends an object $(C,c)$ to $c$ and a morphism $a$ to $f^*(a)|_c$. This functor corresponds to a geometric morphism $g : \Fcal \to \PSh(\Conn(f))$. We consider two classes of sieves on $\Conn(f)$.

\begin{definition} \label{def:complete}
Let $S = \{ a_i : (C_i,c_i) \to (C,c) \}_{i \in I}$ be a sieve on $(C,c)$ in $\Conn(f)$. We say that $S$ is an \textbf{$f$-covering sieve} if its image along $g^*$, $\{ f^*(a_i)|_{c_i} : c_i \to c \}$, is a jointly epimorphic family in $\Fcal$. Further, we say that $S$ is an \textbf{amalgamation covering sieve} if there is a $J$-covering sieve $R$ on $\Ccal$ such that $S$ contains the pullback of $R$ to $(C,c)$, i.e.\ such that for every $a' : C' \to C$ in $R$ and connected component $c' \subseteq f^*(C')$ with $f^*(a)(c') \subseteq c$ we have that the morphism $a' : (C',c') \to (C,c)$ is in $S$.

One can show that any amalgamation covering sieve is $f$-covering. We say that $f$ is \textbf{complete} if every $f$-covering sieve is an amalgamation covering sieve. Further, $f$ is a \textbf{complete spread} if it is both a spread and complete.
\end{definition}

\begin{remark}
\label{rmk:widerdef}
In the setting of toposes, the above concept of completeness was first introduced only for spreads with locally connected domain by Bunge and Funk in \cite{bunge-funk-spreads-I}; they later introduced the notion for arbitrary geometric morphisms with locally connected domain in \cite{bunge-funk}. They prove that the definition above is independent of the site chosen to present $\Ecal$, just as was the case for spreads. The notion of complete spreads in topos theory was inspired by complete spreads in topology, as introduced by Fox \cite{fox}. In this paper, complete geometric morphisms will always have locally connected domain, but note that a wider definition of complete spreads appears in \cite{bunge-funk-hyperpure-journal}; by Proposition 2.2 there, that definition coincides with the one here for geometric morphisms with locally connected domain.
\end{remark}

\begin{theorem} \label{thm:pure-complete-spread-factorization-general}
Let $f : \Fcal \to \Ecal$ be a geometric morphism, with $\Fcal$ locally connected. Then $f$ has a unique factorization as a pure geometric morphism followed by a complete spread, and a unique factorization as a pure surjection followed by a spread. Further, $f$ is complete if and only if its pure part is a surjection (so in this case, the two factorizations coincide). 
\end{theorem}

The theorem is due to Bunge and Funk. Since we will produce explicit reconstructions in the cases of interest in the following subsections, we give references to their book \cite{bunge-funk} in place of a proof. The existence and uniqueness of the (pure, complete spread) factorization are provided in Theorem 2.4.8, while the (pure surjection, spread) factorization appears in Theorem 3.1.16. The (pure surjection, spread) factorization is built by taking the (pure, complete spread) factorization first, and then writing the pure part as a pure surjection followed by a pure inclusion. Because inclusions are spreads, the composition of a pure inclusion followed by a complete spread is a spread. The result that $f$ is complete if and only if its pure part is a surjection is Theorem 3.5.3.

To give some intuition regarding complete geometric morphisms, we include an example and a result for the special case of maps between topological spaces.

\begin{example}
Consider the inclusion $i : W \subseteq \Rbb^2$, with $W$ the unit circle minus the point $(1,0)$. The induced geometric morphism is an inclusion, in particular a spread. We claim that it is not complete. To see this, we take as basis for the topology on $\Rbb^2$ the open balls with radius at most $1/2$. In the notations of Definition \ref{def:complete}, these open balls and the inclusions between them form the site $(\Ccal,J)$. The category $\Conn(i)$ consists of the pairs $(U,c)$, where $U$ is an open ball of radius at most $1/2$ and $c \subseteq W \cap U$ a component. Take a pair $(U,c)$ with $U$ containing $(1,0)$ and an open set $V$ such that $V \cap W = c$. Then the sieve generated by $(V,c) \to (U,c)$ is an $i$-covering sieve. However, it is not an amalgamation covering sieve, because for any amalgamation covering sieve must contain a pair $(U',c')$ with $(1,0) \in U' \subseteq U$ and $c' \subseteq c$.
\end{example}

\begin{proposition}
Let $X \subseteq Y$ be a subspace, with $X$ locally connected. If each connected component $X'$ of $X$ is closed in $Y$, then the geometric morphism induced by $X \subseteq Y$ is complete. The converse holds if we assume that in $Y$ all points are closed.
\end{proposition}
\begin{proof}
Write $i$ for the inclusion of $X$ into $Y$. We take the canonical site of open subsets for $\Sh(Y)$. In the notations of Definition \ref{def:complete}, an object of $\Conn(i)$ is then a pair $(U,c)$ with $U \subseteq Y$ open and $c \subseteq U \cap X$ a connected component.

Suppose that every connected component of $X$ is closed in $Y$. Let $f : \Sh(X) \to \Sh(Y)$ be the induced geometric morphism. Let $S = \{ (U_i,c_i) \to (U,c) \}_{i \in I}$ be an $f$-covering sieve. We claim that it is an amalgamation covering sieve as well. For each $x \in c$, take a pair $(U_i,c_i)$ in $S$ with $x \in c_i$. Since $c_i$ is open in $X \cap U_i$, we can take an open subset $V_x \subseteq U_i$ such that $V_x \cap X = c_i$. Note that $(V_x,c_i)$ is still contained in $S$. Now consider the covering sieve $R$ on $U$ generated by the inclusions $V_x \to U$ for $x \in c$ and the inclusion $U - c \to U$ (to show that $c$ is closed in $U$, use that $c$ is clopen in $X' \cap U$ for some connected component $X'$ of $X$, and that in turn $X' \cap U$ is closed in $U$). The pullback of $R$ to $(U,c)$ is the sieve generated by the inclusions $(V_x,V_x \cap X) \to (U,c)$, and this pullback sieve is contained in $S$. So $S$ is indeed an amalgamation covering sieve.

Conversely, suppose that the induced geometric morphism $f$ is complete. Take a connected component $X'$ of $X$ and an element $y \in \overline{X'}-X'$, with $\overline{X'}$ the closure of $X'$ in $Y$. Because $y$ is closed in $Y$ and $y \notin X$, we can consider the $f$-covering sieve generated by $(Y-\{y\},X') \to (Y,X')$. We claim that this is not an amalgamation covering sieve, which gives a contradiction. To see this, take an arbitrary covering sieve $R$ on $Y$. Then $R$ contains an inclusion $V \to Y$ with $y \in V$. Because $y \in \overline{X'}$, we see that $V \cap X' \neq \varnothing$. So if $S$ contains the pullback of $R$, then $S$ must contain a morphism $(V,c') \to (U,X')$ for some connected component $c' \subseteq V \cap X'$, which leads to a contradiction. As a result, $S$ is not an amalgamation covering sieve.
\end{proof}

\begin{remark}
We would be remiss not to also mention the (\textit{pure, entire}) factorization described by Johnstone in \cite[C3.4]{Ele}. A geometric morphism is \textit{entire} if it is localic and the corresponding internal locale is compact and zero-dimensional. For comparison, under the wider definition of complete spreads referenced in Remark \ref{rmk:widerdef}, any entire geometric morphism is a complete spread, but not conversely.

The (pure, entire) factorization of a morphism $f:\Fcal \to \Ecal$ is obtained by taking the intermediate topos to be the topos of internal sheaves on (the zero-dimensional locale dual to) the subframe of $f_*(\Omega_{\Fcal})$ generated by $f_*(2_{\Fcal})$. This is typically different from the factorizations we consider here. For example, for an infinite set $X$ viewed as a discrete space, the geometric morphism $\Sh(X) \to \Set$ is a complete spread, but its (pure, entire) factorization has sheaves on the Stone-\v{C}ech compactification of $X$ as the intermediate topos. The reason we do not extensively consider this factorization system in this paper is exactly the reason illustrated by that example: the intermediate topos in this factorization is rarely a presheaf topos, even for an essential geometric morphism between presheaf toposes. We leave deeper consideration of entire morphisms to future work.
\end{remark}

\subsection{The essential case}

The following proposition is the dual of Proposition \ref{prop:finfib}.

\begin{proposition}
\label{prop:iniopf}
Let $f : \PSh(\Ccal) \to \PSh(\Dcal)$ be an essential geometric morphism induced by a functor $F : \Ccal \to \Dcal$. Then $F$ has a factorization as an initial functor followed by a discrete opfibration, namely
\begin{equation*}
\Ccal \to \int^{\Dcal}g_!(1) \to \Dcal,
\end{equation*}
where $g:\PSh(\Ccal\op) \to \PSh(\Dcal\op)$ is the essential geometric morphism induced by $F\op$. This is the unique such factorization up to equivalence of the intermediate category. Further, the induced factorization of $f$ coincides with the (pure, complete spread) factorization of $f$.
\end{proposition}
\begin{proof}
The unique factorization of a functor into an initial functor followed by a discrete opfibration is due to Street and Walters \cite{street-walters}. Explicitly, we can obtain this factorization by applying the factorization from Proposition \ref{prop:finfib} to 
\begin{equation*}
F\op: \Ccal\op\to \Dcal\op
\end{equation*}
and then dualizing; recall that a functor is final if and only if its opposite functor is initial, and similarly a functor is a discrete fibration if and only if its opposite is a discrete opfibration. We are also using that, identifying a copresheaf $Y$ on $\Dcal$ with a presheaf on $\Dcal\op$, we have (or may define),
\begin{equation*}
\int^{\Dcal}Y \cong \left(\int_{\Dcal\op} Y \right)\op.
\end{equation*}
That this produces the (pure, complete spread) factorization at the level of geometric morphisms is given as Example 2.16(2) in \cite{bunge-funk-spreads-I}.
\end{proof}

Given this construction, we can immediately dualize the results of Section \ref{sec:tc-etale} to get the corresponding results for the (pure, complete spread) factorization of a geometric morphism $f : \Setswith{M} \to \Setswith{N}$ induced by a semigroup homomorphism $\phi : M \to N$. We first introduce the dual of Definition \ref{dfn:rfactor}.

\begin{dfn}
\label{dfn:lfactor}
A subset $S$ of a monoid $M$ is called \textbf{left-factorable} if whenever $x \in M$ and $y \in S$ with $xy \in S$, then $x \in S$. For an arbitrary subset $T$, we define $\langle\langle{T}\rangle_M \subseteq M$ to be the smallest left-factorable submonoid of $M$ containing $T$, and call this the submonoid of $M$ \textbf{left-factorably generated} by $T$.
\end{dfn}

For a geometric morphism $f : \Setswith{M} \to \Setswith{N}$ induced by a semigroup homomorphism $\phi : M \to N$, we constructed in Section \ref{sec:tc-etale} a factorization
\begin{equation*}
\begin{tikzcd}
\Setswith{M} \ar[r,"{\text{tc surj.}}"] & \Setswith{\langle{\phi(M)}\rangle\rangle_{eNe}} \ar[d,"{\text{tc incl.}}"] & & \\
& \Setswith{\langle{\phi(M)}\rangle\rangle_N} \ar[r,"{\text{tc incl.}}"] & \PSh(\int_N 1\otimes_M eN) \ar[d,"{\text{\'etale}}"] & \\
& & \Setswith{N}.
\end{tikzcd}
\end{equation*}

If we apply this factorization to $\phi\op : M\op \to N\op$ and then take opposites, then we get the factorization
\begin{equation} \label{eq:factorization-parts}
\begin{tikzcd}
\Setswith{M} \ar[r,"{\text{pure surj.}}"] & \Setswith{\langle\langle{\phi(M)}\rangle_{eNe}} \ar[d,"{\text{pure incl.}}"] & & \\
& \Setswith{\langle\langle{\phi(M)}\rangle_N} \ar[r,"{\text{pure incl.}}"] & \PSh(\int^N eN \otimes_M 1)
 \ar[d,"{\text{complete spread}}"] & \\
& & \Setswith{N}.
\end{tikzcd}
\end{equation}
Here we have made use of the following equalities:
\begin{gather*}
\left(\langle{\phi(M\op)}\rangle\rangle_{N\op} \right)\op = \langle\langle{\phi(M)}\rangle_N,\qquad
\left(\langle{\phi(M\op)}\rangle\rangle_{(eNe)\op} \right)\op = \langle\langle{\phi(M)}\rangle_{eNe}, \\
\left(\int_{N\op} Y\right)\op = \int^N Y,\qquad 1 \otimes_{M\op} X = X \otimes_M 1,
\end{gather*}
with $Y$ a left $N$-set and $X$ a right $M$-set. We also observe that the (surjection, inclusion) factorization is self-dual. 

We can deduce from this the dual to Corollary \ref{cor:tc}.

\begin{crly}
\label{crly:pure-essential-case}
A geometric morphism $f : \Setswith{M} \to \Setswith{N}$ induced by a semigroup homomorphism $\phi:M \to N$ is pure if and only if $\langle\langle{\phi(M)}\rangle_{N} = N$.
\end{crly}

Dualizing the argument we saw in Section \ref{sec:tc-etale}, we deduce that essential complete spreads induced by semigroup homomorphisms are surjective. This produces the following dual to Theorem \ref{thm:etale}.

\begin{theorem}
\label{thm:complete-spread}
Let $f : \Setswith{M} \to \Setswith{N}$ be an essential geometric morphism induced by a monoid homomorphism $\phi : M \to N$. Then the following are equivalent:
\begin{enumerate}
\item $f$ is a complete spread;
\item $\phi$ is injective, $\phi(M) \subseteq N$ is left-factorable and for any $n \in N$ there is some $v \in N^\rtimes$ such that $vn \in \phi(M)$.
\end{enumerate}
More generally, if $\phi$ is merely a semigroup homomorphism, then $f$ is \'etale if and only if the monoid homomorphism part of $\phi$ satisfies the conditions above, and the inclusion $eNe \subseteq N$ induces an equivalence, where $e = \phi(1)$.
\end{theorem}

The condition that for any $n \in N$ there is some $v \in N^\rtimes$ such that $vn \in \phi(M)$, corresponds to the essential inclusion $\Setswith{\langle\langle{\phi(M)}\rangle_N} \to \PSh(\int^N Ne \otimes_M 1)$ being an equivalence. Further, the condition that $\phi$ is injective with $\phi(M) \subseteq N$ left-factorable, corresponds to the condition that the pure surjection part $\Setswith{M} \to \Setswith{\langle\langle{\phi(M)}\rangle_N}$ is an equivalence. The geometric morphisms such that the pure surjection part is trivial are precisely the spreads, so:

\begin{proposition}
Let $f : \Setswith{M} \to \Setswith{N}$ be an essential geometric morphism induced by a monoid homomorphism $\phi : M \to N$. Then $f$ is a spread if and only if $\phi$ is injective and $\phi(M) \subseteq N$ is left-factorable.
\end{proposition}

More generally, if $f$ is induced by a semigroup homomorphism $\phi: M \to N$, then the pure surjection part is given by $\PSh(M) \to \PSh(\langle\langle{\phi(M)}\rangle_{eNe})$, as shown in \eqref{eq:factorization-parts}. Again, $f$ is a spread if and only if the pure surjection part is trivial, so if and only if $\phi$ is injective and $\phi(M) \subseteq eNe$ is left-factorable.

Finally, we give an updated version of Example \ref{eg:etale-geometric-morphisms}.

\begin{example} \label{eg:etale-and-complete-spread} \ 
\begin{enumerate}
\item For $H \subseteq G$ an inclusion of groups, we have that the induced geometric morphism $\Setswith{H} \to \Setswith{G}$ is both \'etale and a complete spread.
\item Consider the monoid $\Zbb_p^\ns$ of nonzero $p$-adic integers under multiplication. Then the inclusion $\mathbb{N} \to \Zbb_p^\ns,~ k \mapsto p^k$ induces an essential geometric morphism $\Setswith{\Nbb} \to \Setswith{\Zbb_p^\ns}$ that is both \'etale and complete spread.
\end{enumerate}
\end{example}

In general, an \'etale geometric morphism $f: \PSh(M) \to \PSh(N)$ is not necessarily a complete spread (and vice versa). In Subsection \ref{ssec:matrix}, we give an extreme example of a geometric morphism which is both terminal-connected and a complete spread (so it has trivial \'etale part), and dually, an example of a morphism which is both pure and \'etale (so the complete spread part is trivial).

\subsection{The general case}

As we alluded to earlier, the (pure, complete spread) factorization works for general geometric morphisms with locally connected domain. We follow the construction of \cite[Proposition 2.11]{bunge-funk-spreads-I} in our special case. A geometric morphism $f: \Setswith{M} \to \Setswith{N}$ with $f^*(X) = X \otimes_N A$ has (pure, complete spread) factorization of the following form,
\begin{equation} \label{eq:pure-complete-spread-factorization}
\begin{tikzcd}
\PSh(M) \ar[r,"{\eta}"] & \PSh(\int^N A \otimes_M 1) \ar[r,"{\pi}"] & \PSh(N),
\end{tikzcd}
\end{equation}
where $\pi$ is the essential geometric morphism induced by the discrete opfibration
\begin{equation*}
\begin{tikzcd}
\int^N A \otimes_M 1 \ar[r] & N
\end{tikzcd}
\end{equation*}
and $\eta$ is the geometric morphism determined by the flat functor
\begin{equation*}
\begin{tikzcd}
\int^N A \otimes_M 1 \ar[r,"{V}"] & \Setswith{M}
\end{tikzcd}
\end{equation*}
which sends an element $c \in A \otimes_M 1$ to the corresponding component of $A$ as right $M$-set.

From the above discussion, we deduce:
\begin{proposition} \label{prop:complete-spreads-are-essential}
Let $f : \PSh(M) \to \PSh(N)$ be a complete spread, for monoids $M$ and $N$. Then $f$ is essential.
\end{proposition}

So complete spreads $\PSh(M) \to \PSh(N)$ are completely characterized by Theorem \ref{thm:complete-spread}. Note that from the work of Bunge and Funk \cite[Proposition 2.11]{bunge-funk-spreads-I} it follows more generally that any complete spread between presheaf toposes is essential. We shall see that this is not the case for pure geometric morphisms.

In the setting of this paper, it is natural to ask when the intermediate topos $\PSh(\int^N A \otimes_M 1)$ in \ref{eq:pure-complete-spread-factorization} is equivalent to $\PSh(B)$ for some monoid $B$. 

\begin{proposition} \label{prop:intermediate-topos-is-monoid-topos}
Let $f : \PSh(M) \to \PSh(N)$ be the geometric morphism determined by a $\fllr{N}{M}$-set $A$. Then $\PSh(\int^N A \otimes_M 1)\simeq \PSh(B)$ for some monoid $B$ if and only if there is some $c \in A \otimes_M 1$ such that for every $c' \in A \otimes_M 1$ there is some $v \in N^\rtimes$ such that $vc' = c$. In this case, the (pure, complete spread) factorization has the following more concrete description. We can take $B = \{ n \in N : nc = c \}$. Let $A' \subseteq A$ be the component of $A$ corresponding to $c$. Then the left $N$-action on $A$ restricts to a left $B$-action on $A'$, and $A'$ is flat as left $B$-set. The geometric morphism $\eta : \PSh(M) \to \PSh(B)$ is determined by the $\fllr{B}{M}$-set $A'$, and the (essential) geometric morphism $\pi: \PSh(B) \to \PSh(N)$ is induced by the monoid inclusion $B \subseteq N$. 
\end{proposition}
\begin{proof}
By Lemma \ref{lem:retracted}, there is an equivalence $\PSh(\int^N A \otimes_M 1) \simeq \PSh(B)$ if and only if there is an object $c$ of $\int^N A \otimes 1$ of which every object in $\int^N A \otimes_M 1$ is a retract. Now $c'$ is a retract of $c$ if and only if there are $u,v \in N$ with $uv = 1$ and $vc' = c$ and $uc = c'$ (the last equation follows from the first two). The stated conditions follow, with $B$ the endomorphism monoid of $c$ in $\int^N A \otimes_M 1$.

The equivalence $\PSh(B) \to \PSh(\int^N A \otimes_M 1)$ is induced by the inclusion of $B$ as a full subcategory of $\int^N A \otimes_N 1$ (as the endomorphism monoid of $c$). If we compose this with the projection to $N$, then we get the monoid inclusion $B \subseteq N$. By the construction of the (pure, complete spread) factorization, the complete spread part $\pi : \PSh(B) \to \PSh(N)$ is the geometric morphism induced by this monoid inclusion $B \subseteq N$. Meanwhile, to compute the $\fllr{B}{M}$-set $A'$ inducing $\eta$, we identify $B$ with the representable presheaf in $\PSh(\int^N A \otimes_M 1)$ corresponding to the element $c$. By construction, this is sent by $\eta$ to the $M$-connected component of $A$ corresponding to $c$. So the corresponding flat functor $V : B \to \PSh(M)$ is the one given by this component, with the action of $B$ coming from the images under $\eta$ of the endomorphisms of $c$.
\end{proof}

The form of the pure part should come as no surprise after the following characterization.

\begin{proposition} \label{prop:characterization-pure}
Let $f : \Setswith{M} \to \Setswith{N}$ be the geometric morphism corresponding to the $\fllr{N}{M}$-set $A$. Then $f$ is pure if and only if $A$ is connected as a right $M$-set.
\end{proposition}
\begin{proof}
Since $\Setswith{M}$ and $\Setswith{N}$ are locally connected, $f$ is pure if and only if $f_*$ preserves small coproducts. We have $f_* \simeq \Hom_M(A,-)$, which preserves small coproducts if and only if $A$ is connected as right $M$-set.
\end{proof}

We can apply this in particular to a geometric morphism $f$ induced by a semigroup morphism $\phi: M \to N$, where $A = Ne$ with $e = \phi(1)$ as usual. We then find that $f$ is pure if and only if $Ne$ is connected as right $M$-set. This gives an alternative route to Corollary \ref{crly:pure-essential-case}.

For the sake of completeness, we also characterize spreads and injections by extending the argument we saw in Proposition \ref{prop:characterization-localic}.

\begin{scholium} \label{schl:characterization-spread}
Let $f : \Setswith{M} \to \Setswith{N}$ be the geometric morphism corresponding to the $\fllr{N}{M}$-set $A$. Then $f$ is a spread if and only if $M$ is a retract of some connected component of $A$, as a right $M$-set. More strongly, $f$ is an injection if and only if $M$ is a retract of $A$.
\end{scholium}
\begin{proof}
We simply replace `subobject' with `complemented subobject' and `object', respectively, in the proof of Proposition \ref{prop:characterization-localic}.
\end{proof}
Note that $M$ is a retract of a connected component of $A$ if and only if this connected component generates $\Setswith{M}$, see \cite[II, Theorem 3.16]{MAC}.

We can further combine Proposition \ref{prop:characterization-pure} and Scholium \ref{schl:characterization-spread} to give a characterization of pure spreads, or equivalently by the comments following Theorem \ref{thm:pure-complete-spread-factorization-general}, pure inclusions.

\begin{crly} \label{crly:characterization-pure-spread}
Let $f : \Setswith{M} \to \Setswith{N}$ be the geometric morphism corresponding to the $\fllr{N}{M}$-set $A$. Then $f$ is a pure spread (or equivalently, a pure inclusion) if and only if $A$ is connected as a right $M$-set and has $M$ as a retract.
\end{crly}

\begin{example}
Consider the $\fllr{\Nbb}{\Zbb}$-set $\Zbb$, with the left and right action given by addition. Here $\Zbb$ is connected as a right $\Zbb$-set and there is an epimorphism of right $\Zbb$-sets $\Zbb \to \Zbb$ (the identity map). So the geometric morphism $\PSh(\Zbb) \to \PSh(\Nbb)$ described by the $\fllr{\Nbb}{\Zbb}$-set $\Zbb$ is a pure inclusion. 

More generally, let $\phi: N \to Z$ be a monoid map such that $Z$ is flat as a left $N$-set. Equivalently, $\phi$ is flat as a functor, see \cite[4.7]{benabou}. Then the geometric morphism $\PSh(Z) \to \PSh(N)$ described by the $\fllr{N}{Z}$-set $Z$ is a pure inclusion.
\end{example}

\begin{remark}
Since we have established that the (terminal-connected, \'{e}tale) factorization is dual to the (pure, complete spread) factorization, one might wonder why we did not simply dualize the construction of the latter factorization in this section in order to obtain the former factorization for arbitrary geometric morphisms. The reason is that `reversing' an $\fllr{N}{M}$-set produces a $(M\op,N\op]$-set, and hence a distribution going in the opposite direction. This distribution can be factorized via the dual of the construction above, but the result cannot in general be dualized back to a factorization of the original geometric morphism. Alternatively, one can directly observe that the conditions of Lemma \ref{lem:tccond} and Proposition \ref{prop:characterization-pure} are not dual to one another.
\end{remark}

\section{Comparing \'etale and complete spread geometric morphisms}
\label{sec:compare}

We have seen several examples of geometric morphisms which are both \'{e}tale and complete spreads. In this section, we examine the relationship between these classes of morphism in more detail, first in general and then applied to our case of interest.

\subsection{Locally constant \'etale morphisms}

By definition, objects of a topos $\Ecal$ correspond (up to equivalence of domain toposes) to \'etale geometric morphisms with codomain $\Ecal$. The most basic kind of \'etale maps are the constant \'etale maps. These correspond to the objects $A$ with $A = \bigsqcup_{i \in I} 1$ a disjoint union of copies of the terminal object; these are called the \textbf{constant objects}, and when $\Ecal$ is a Grothendieck topos they can equivalently be expressed as being of the form $p^*(I)$, where $p$ is the global sections morphism of $\Ecal$. The corresponding \'etale geometric morphism is equivalent to the canonical projection $\bigsqcup_{i \in I} \Ecal \to \Ecal$, where $\bigsqcup_{i \in I} \Ecal$ denotes the coproduct of $I$ copies of $\Ecal$, in the category of toposes.

For objects $A$ and $U$ we say that $A$ is \textbf{trivialized} by $U$ if there is a commutative diagram
\begin{equation*}
\begin{tikzcd}
A \times U \ar[rr,"{\psi}","{\cong}"'] \ar[rd] & & \bigsqcup_{i \in I} U \ar[ld] \\
 & U &
\end{tikzcd}
\end{equation*}
with $\psi$ an isomorphism, where the diagonal maps are the evident projections.

\begin{definition}
An object $A$ of a topos $\Ecal$ is said to be \textbf{locally constant} if there is a family of objects $\{U_k\}_{k \in K}$ whose morphisms to the terminal object are jointly epimorphic such that $A$ is trivialized by each of the $U_k$. An \'{e}tale geometric morphism with codomain $\Ecal$ is called \textbf{locally constant \'etale} if it is (up to equivalence of the domain) of the form $\Ecal/A \to \Ecal$ with $A$ locally constant.
\end{definition}

For a topological space $X$, the locally constant \'etale geometric morphisms with codomain $\Sh(X)$ are those of the form $\Sh(Y) \to \Sh(X)$ induced by a covering map $Y \to X$.

We have seen that all \'{e}tale morphisms over presheaf toposes are induced by discrete fibrations, and dually that complete spreads are induced by discrete opfibrations.

\begin{definition}
A functor $F:\Ccal \to \Dcal$ is a \textbf{discrete bifibration} if it is both a discrete fibration and a discrete opfibration.
\end{definition}

Locally constant \'etale geometric morphisms to presheaf toposes are characterized as in the following proposition. The equivalence $(1) \Leftrightarrow (4)$ is due to Bunge and Funk \cite[Corollary 7.9]{bunge-funk-spreads-II}. 

\begin{proposition}
\label{prop:lcpshf}
For a presheaf $A$ on a small category $\Dcal$, the following are equivalent:
\begin{enumerate}
\item $A$ is locally constant as an object of $\PSh(\Dcal)$,
\item Given any morphism $g:D' \to D$ of $\Dcal$, $A(g)$ is an isomorphism,
\item The discrete fibration $\int_{\Dcal} A \to \Dcal$ is a discrete bifibration,
\item The \'{e}tale geometric morphism $\PSh(\Dcal)/A \to \PSh(\Dcal)$ is a complete spread.
\end{enumerate}
\end{proposition}
\begin{proof}
The equivalence $(1) \Leftrightarrow (2)$ follows from a more general result by Leroy in \cite[Proposition 2.2.1]{leroy}; we give a simplified argument in this special case.

$(1) \Rightarrow (2)$ Given a trivialising family $\{U_k\}_{k \in K}$ for $A$ which is jointly epimorphic over $1$, there must in particular be some $k \in K$ such that $U_k(D) \neq \emptyset$, whence also $U_k(D') \neq \emptyset$. Consider the naturality diagram,
\begin{equation*}
\begin{tikzcd}
A(D) \times U(D) \ar[rr,"{\psi_{D}}","{\cong}"'] \ar[rd] \ar[dd, "A(g) \times U(g)"'] & & \bigsqcup_{i \in I} U(D) \ar[ld] \ar[dd] \\
 & U(D) \ar[dd] & \\
A(D') \times U(D') \ar[rr,"{\psi_{D'}}","{\cong}"', near start, crossing over] \ar[rd] & & \bigsqcup_{i \in I} U(D') \ar[ld] \\
 & U(D'), &
\end{tikzcd}
\end{equation*}
where $I$ is fixed. Let $u \in U(D)$ and $u' := U(g)(u)$. Given $x,y \in A(D)$ with $A(g)(x) = A(g)(y) = z$, say, we write $\psi_D(x,u) = (i,u)$ and $\psi_D(y,u) = (j,u)$ for certain indices $i,j \in I$. We have $\psi_{D'}(z,u') = (i,u')$, and similarly $\psi_{D'}(z,u') = (j,u')$, so $i=j$ and as a result $x = y$. Conversely, given $z \in A(D')$, we write $\psi_{D'}(z,u') = (i,u')$ for some index $i \in I$. Now take $x \in A(D)$ with $\psi_D(x,u) = (i,u)$. It follows that $A(g)(x) = z$. 

$(2) \Rightarrow (1)$ For $A$ satisfying the given condition, let us take $\{\yon(D)\}_{D \in \ob(\Dcal)}$ as our set of trivializing objects. We have an isomorphism $\yon(D) \times A \to \coprod_{a \in A(D)} \yon(D)$ which at an object $D'$ is defined by
\begin{align*}
A(D') \times \Hom(D',D) &\to \coprod_{a \in A(D)} \Hom(D',D) \\
(a,g) & \mapsto (A(g)(a),g),
\end{align*}
which is a bijection since $A(g)$ is.

$(2) \Leftrightarrow (3)$ We already know that $\int_{\Dcal} A \to \Dcal$ is a discrete fibration, so it suffices to check whether it is a discrete opfibration. Given an object $(D,a)$ of $\int_{\Dcal} A$ and a morphism $g:D \to D'$ in $\Dcal$, we know that $A(g)$ is a bijection, so we have a unique lifting of $g$ to the morphism $g: (D,a) \to (D',A(g)^{-1}(a))$, as required. Conversely, if the projection is a discrete opfibration, then each $a' \in A(D')$ has a unique pre-image along $A(g)$ for any $g$, so $A(g)$ is a bijection, as required.

$(3) \Leftrightarrow (4)$ This follows from the fact that a functor $F$ induces a complete spread if and only if it is a discrete opfibration up to equivalence of the domain category, after verifying that $F$ being a discrete fibration forces this equivalence to be an isomorphism.
\end{proof}

This result applies in particular to $\Ecal = \PSh(N)$. We can combine it with the characterizations from the last section to deduce the following.
\begin{corollary}
\label{crly:lconstant-char}
Let $f : \Setswith{M} \to \Setswith{N}$ be an essential geometric morphism induced by a monoid homomorphism $\phi : M \to N$. Then the following are equivalent:
\begin{enumerate}
\item $f$ is locally constant \'etale;
\item $\phi$ is injective, $\phi(M) \subseteq N$ is both left-factorable and right-factorable, and for any $n \in N$ there are elements $u \in N^{\ltimes}$, $v \in N^{\rtimes}$ such that $nu \in \phi(M)$ and $vn \in \phi(M)$.
\end{enumerate}
More generally, if $\phi$ is merely a semigroup homomorphism, then $f$ is locally constant \'etale if and only if the monoid homomorphism part of $\phi$ satisfies the conditions above, and the inclusion $eNe \subseteq N$ induces an equivalence, where $e=\phi(1)$.
\end{corollary}

It follows from this corollary that if $N$ is commutative, then $f$ is \'etale if and only if it is a complete spread, if and only if it is locally constant \'etale. Independently, if $N$ is a group, then any inclusion of a subgroup into $N$ induces a locally constant \'etale morphism.

\subsection{\'Etale geometric morphisms with fixed codomain}
\label{ssec:fixed-codomain-etale}

While the abstract characterizations of subsemigroups inducing \'etale geometric morphisms and complete spreads from the previous sections are useful for recognizing these properties, classifying such morphisms of the form $f : \Setswith{M} \to \Setswith{N}$ for a fixed monoid $N$ can still be challenging. Here we explore a different approach in terms of $N$-sets.

We can directly use Lemma \ref{lem:retracted} to identify objects $X$ of $\PSh(N)$ such that 
\begin{equation*}
\PSh(N)/X \simeq \PSh(M)
\end{equation*}
for some monoid $M$; namely, this happens if there is some object $x \in \int_N X$ of which every object is a retract. Given an element $y \in X$, $y$ is a retract of $x$ as an object of $\int_N X$ if and only if $\exists u \in N^\ltimes$ with $x=yu$. Letting $v$ be the right inverse of $u$, this can be expressed in the following diagram:
\begin{equation*}
\begin{tikzcd}
x \ar[loop left, "vu"] \ar[r,"{u}"'] & y \ar[loop right, "\id"] \ar[l,"{v}"',bend right]
\end{tikzcd}.
\end{equation*}

\begin{definition}
Let $N$ be a monoid and let $X$ be a right $N$-set. An element $x \in X$ will be called a \textbf{strong generator} if for all $y \in X$ there is an element $u \in N^\ltimes$ such that $yu = x$. 
\end{definition}

Note that a strong generator is in particular a generator, so a right $N$-set that admits a strong generator i s cyclic. If $x$ is a strong generator and $u \in N^\ltimes$, then $xu$ is again a strong generator.

\begin{theorem} \label{thm:etale-construction}
Let $N$ be a monoid. Fix a right $N$-set $X$ and a strong generator $x \in X$, and set $N_x = \{ n \in N : xn = x \}$. Then the inclusion $N_x \subseteq N$ induces an \'etale geometric morphism
\begin{equation*}
\begin{tikzcd}
\Setswith{N_x} \ar[r] & \Setswith{N}. 
\end{tikzcd}
\end{equation*}
Conversely, every \'etale geometric morphism $\Setswith{M}\to \Setswith{N}$ is of this form (up to precomposition with an equivalence).  
\end{theorem}
\begin{proof}
The \'etale geometric morphisms with codomain $\Setswith{N}$ are precisely the geometric morphisms $\PSh(\int_N X) \to \Setswith{N}$ for some right $N$-set $X$. Further, from the above we see that $\PSh(\int_N X) \simeq \Setswith{M}$ for some monoid $M$ if and only if $X$ has a strong generator $x$, and in this case we can take $M$ to be $N_x$, which is the endomorphism monoid of $x$ in $\int_N X$. 
\end{proof}

Alternatively, one direction of the statement can be deduced from Theorem \ref{thm:etale}.

\begin{remark} \label{rem:induced-by-monoid-map}
If $f: \Setswith{M} \to \Setswith{N}$ is an \'etale geometric morphism, we already saw $f$ is surjective. This means that if $f$ is induced by a semigroup homomorphism $\phi : M \to N$ such that for $e=\phi(1)$ the inclusion $eNe \subseteq N$ induces an equivalence. So after replacing $N$ with the Morita equivalent monoid $eNe$, we can assume that $\phi$ is a monoid map.

Using Theorem \ref{thm:etale-construction} we see that if the semigroup map $\phi : M \to N$ induces an \'etale geometric morphism $f$, then we can also replace $M$ by a Morita equivalent monoid $M'$ such that there is a monoid morphism $\phi' : M' \to N$ inducing $f$. 

This does not work for arbitrary surjective geometric morphisms. For example, take $N$ a monoid with a nontrivial idempotent $e \in N$, and consider the semigroup map $\phi : 1 \to N$ with $\phi(1)=e$. This induces a geometric morphism $f: \Set \to \Setswith{N}$ which is surjective whenever the inclusion $eNe \subseteq N$ induces a Morita equivalence. However, there are no monoids Morita equivalent to $1$ (other than $1$ itself), so it is impossible for $f$ to be induced by a monoid map $M' \to N$ for some monoid $M'$. 
\end{remark}

\begin{example}
If we have a monoid $N$, a right $N$-set $X$, and strong generators $x,x' \in X$, then $N_x$ and $N_{x'}$ are Morita equivalent, since both $\Setswith{N_x}$ and $\Setswith{N_{x'}}$ are equivalent to $\PSh(\int_N X)$. We now show with an example that $N_x$ and $N_{x'}$ are not necessarily isomorphic.

Let $N = \langle{u,v,t: uv=1, t^2 = 1}\rangle$ and $X = \{a,b\}$. Consider the right $N$-action on $X$ defined on generators as $a \cdot u = a \cdot v = a \cdot t = b \cdot t = b$ and $b \cdot u = b\cdot v = a$.
\begin{equation*}
\begin{tikzcd}
a \ar[r,bend left,"{u,v,t}"] & b \ar[l,bend left,"{u,v}"] \ar[loop,"{t}", out=30, in=330,looseness=5]
\end{tikzcd}
\end{equation*}
Clearly, both $a$ and $b$ are strong generators, so the endomorphism monoids $\End(a)$ and $\End(b)$ in $\int_N X$ are Morita equivalent, since both present the topos $\PSh(N)/X$.

We shall show that $\End(a) \not\cong \End(b)$ by examining the idempotents of the involved monoids. In $N$, any idempotent can be reduced to a form in which $t$ occurs at most once, and from there that any idempotent can be written as either $v^ku^k$ or $v^ktu^k$ for some $k \in \Nbb$. The idempotents lying in $\End(a)$ are all those of the form $v^ku^k$, plus those of the form $v^{2i+1}tu^{2i+1}$ for some $i \in \Nbb$. On the other hand, $\End(b)$ contains the idempotent of the form $v^ku^k$ and $v^{2j}tu^{2j}$ for $j \in \Nbb$. In $\End(a)$ there is a idempotent $e\neq 1$ such that $ef=f$ for any other idempotent $f\neq 1$, namely $e=vu$. In $\End(b)$ such an idempotent does not exist, simply because $t \in \End(b)$. Indeed, from $et=t$ and $e \neq 1$ it would follow that $e=t$, but this idempotent does not qualify because $tvu \neq vu$ (indeed, $tvu$ is not even idempotent).
\end{example}

\subsection{Complete spreads with fixed codomain}

We can dualize the results from the previous subsection to complete spreads.

\begin{definition}
Let $N$ be a monoid and let $Y$ be a left $N$-set. An element $y \in Y$ will be called a \textbf{strong generator} if for all $x \in Y$ there is an element $v \in N^\rtimes$ such that $vx = y$. 
\end{definition}

\begin{theorem} \label{thm:complete-spread-construction}
Let $N$ be a monoid. Fix a left $N$-set $Y$ and a strong generator $y \in Y$, and set $N^y = \{ n \in N : ny = y \}$. Then the inclusion $N^y \subseteq N$ induces a complete spread
\begin{equation*}
\begin{tikzcd}
\Setswith{N^y} \ar[r] & \Setswith{N}. 
\end{tikzcd}
\end{equation*}
Conversely, every complete spread $\Setswith{M}\to \Setswith{N}$ is of this form (up to precomposition with an equivalence).  
\end{theorem}

\begin{remark}
The dual of Remark \ref{rem:induced-by-monoid-map} holds here. If $f : \Setswith{M} \to \Setswith{N}$ is a complete spread, then we can assume that it is induced by a monoid map $M \to N$, either by replacing $N$ with a Morita equivalent monoid (using that $f$ is surjective), or by replacing $M$ with a Morita equivalent monoid (using Theorem \ref{thm:complete-spread-construction}).
\end{remark}

Observe that while we were able to characterize locally constant \'etale geometric morphisms in terms of subsemigroups in Corollary \ref{crly:lconstant-char}, we cannot combine Theorems \ref{thm:etale-construction} and \ref{thm:complete-spread-construction} so easily since they refer to fundamentally different objects (right and left $N$-sets, respectively). We can instead employ the characterization from Proposition \ref{prop:lcpshf} to deduce the following.

\begin{corollary}
\label{crly:lc-generators}
Let $X$ be a right $N$-set with a strong generator $x$. Then the induced \'etale geometric morphism $\PSh(N_x) \to \PSh(N)$ is locally constant \'etale if and only if $N$ acts on $X$ by automorphisms. Dually, if $Y$ is a left $N$-set with a strong generator $y$, then the induced complete spread $\PSh(N^y) \to \PSh(N)$ is locally constant \'etale if and only if $N$ acts on $Y$ by automorphisms.
\end{corollary}

\subsection{A matrix monoid example}
\label{ssec:matrix}

In this subsection, we give an example of a monoid homomorphism such that the induced geometric morphism is both terminal-connected and a complete spread. After dualizing, this additionally gives an example where the induced geometric morphism is both pure and \'etale.

The example is inspired by some of the literature on the Arithmetic Site of Connes and Consani \cite{connes-consani}, \cite{connes-consani-complex-lift}, \cite{arithmtop}, \cite{llb-three}.

For a prime number $p$, consider the monoid
\begin{equation*}
Q_p = \left\{ \begin{pmatrix}
p^n & 0 \\
k & 1
\end{pmatrix} : n \in \Nbb,~ k \in \Zbb  \right\}
\end{equation*}
under matrix multiplication, and the submonoid
\begin{equation*}
F_p = \left\{ \begin{pmatrix}
p^n & 0 \\
k & 1
\end{pmatrix} : n,k \in \Nbb,~ 0 \leq k < p^n  \right\} ~\subseteq~ Q_p.
\end{equation*}

Here we think of $\PSh(F_p)$ as corresponding to the prime $p$ part of Conway's site as introduced in \cite{llb-three}. Further, $Q_p$ is the prime $p$ part of (the opposite of) the $(ax+b)$-monoid, which is related to the study of parabolic $\mathbb{Q}$-lattices, see \cite{connes-consani-complex-lift}. The topos $\PSh(Q_p)$ is the prime $p$ part of the topos associated to the $(ax+b)$-monoid, as studied in \cite[\S 2.5]{arithmtop}.

\begin{proposition}
The monoid $F_p$ is free, with as generators the matrices
\begin{equation*}
\begin{pmatrix}
p & 0 \\
0 & 1
\end{pmatrix},~
\begin{pmatrix}
p & 0 \\
1 & 1
\end{pmatrix},~
\dots
\begin{pmatrix}
p & 0 \\
p-1 & 1
\end{pmatrix}.
\end{equation*} 
\end{proposition}
\begin{proof}
For $n \in \Nbb$, take natural numbers $a_0,\dots,a_{n-1}$ with $0 \leq a_i < p$ for each $i \in \{0,\dots,n-1\}$. We then calculate
\begin{equation*}
\begin{pmatrix}
p & 0 \\
a_{n-1} & 1
\end{pmatrix}
\begin{pmatrix}
p & 0 \\
a_{n-2} & 1
\end{pmatrix} \dots 
\begin{pmatrix}
p & 0 \\
a_1 & 1
\end{pmatrix}
\begin{pmatrix}
p & 0 \\
a_0 & 1
\end{pmatrix} =
\begin{pmatrix}
p^n & 0 \\
\sum_{i=0}^{n-1} a_i p^i & 1 
\end{pmatrix}.
\end{equation*}
The submonoid generated by the given matrices is then free, by the uniqueness of $p$-adic expansions. 
\end{proof}

\begin{proposition}
Let $f : \PSh(F_p) \to \PSh(Q_p)$ be the geometric morphism induced by the inclusion $F_p \subseteq Q_p$. Then $f$ is terminal-connected and a complete spread.
\end{proposition}
\begin{proof}
We first prove that $f$ is terminal-connected. By Corollary \ref{cor:tc}, $f$ is terminal-connected if and only if $\langle{F_p}\rangle\rangle_{Q_p} = Q_p$. Take arbitrary $n$ and $k$, with $n \in \Nbb$ and $k \in \Zbb$. Choose a natural number $r$ such that $p^{n+r} + k \geq 0$ and then choose $m>r$ large enough such that $p^{n+r} + k < p^{m+n}$. We compute
\begin{equation*}
\begin{pmatrix}
p^m & 0 \\
p^r & 1 
\end{pmatrix} \begin{pmatrix}
p^n & 0 \\
k & 1
\end{pmatrix} = \begin{pmatrix}
p^{m+n} & 0 \\
p^{n+r} + k & 1
\end{pmatrix}.
\end{equation*}
The matrices $\begin{pmatrix}
p^m & 0 \\
p^r & 1
\end{pmatrix}$ and $\begin{pmatrix}
p^{m+n} & 0 \\
p^{n+r} + k & 1
\end{pmatrix}$ are both contained in $F_p$. It follows that $\begin{pmatrix}
p^n & 0 \\
k & 1
\end{pmatrix}$ is contained in the right-factorable closure $\langle{F_p}\rangle\rangle_{Q_p}$. Because $n$ and $k$ were arbitrary, we conclude that $\langle{F_p}\rangle\rangle_{Q_p} = Q_p$. So $f$ is terminal-connected.

We now prove that $f$ is a complete spread. By Theorem \ref{thm:complete-spread}, it is enough to show that $F_p \subseteq Q_p$ is left-factorable, and that for any $x \in Q_p$ there is some $v \in Q_p^\rtimes$ such that $vx \in F_p$. To show that $F_p \subseteq Q_p$ is left-factorable, we compute
\begin{equation*}
\begin{pmatrix}
p^n & 0 \\
k & 1 
\end{pmatrix}
\begin{pmatrix}
p^m & 0 \\
l & 1
\end{pmatrix} = 
\begin{pmatrix}
p^{n+m} & 0 \\
kp^m + l & 1 
\end{pmatrix}.
\end{equation*}
We now have to show that if $0 \leq kp^m + l < p^{n+m}$ and $0 \leq l < p^m$, then $0 \leq k < p^n$. We leave it to the reader to verify this. Now take
\begin{equation*}
x = \begin{pmatrix}
p^n & 0 \\
k & 1
\end{pmatrix} \in Q_p.
\end{equation*}
We have to find $v \in Q_p^\rtimes$ such that $vx \in F_p$. Note that
\begin{equation*}
Q_p^\rtimes = Q_p^\times = \left\{ \begin{pmatrix}
1 & 0 \\
z & 1
\end{pmatrix} : z \in \Zbb \right\}
\end{equation*}
and
\begin{equation*}
\begin{pmatrix}
1 & 0 \\
z & 0
\end{pmatrix}
\begin{pmatrix}
p^n & 0 \\
k & 1 
\end{pmatrix} =
\begin{pmatrix}
p^n & 0 \\
zp^n + k & 1
\end{pmatrix}.
\end{equation*}
So we need to find an integer $z$ such that $0 \leq zp^n + k < p^n$. There is a unique such $z$, namely the smallest $z$ with $0 \leq zp^n + k$. 
\end{proof}

Because terminal-connected is dual to pure, and \'etale is dual to being a complete spread, we can dualize to get an example of a pure and \'etale geometric morphism.

\begin{corollary}
Let $g : \PSh(F_p\op) \to \PSh(Q_p\op)$ be the geometric morphism induced by the inclusion $F_p\op \subseteq Q_p\op$. Then $f$ is pure and \'etale. 
\end{corollary}

Now consider the inclusion $F_p \times F_p\op \subseteq Q_p \times Q_p\op$, and let $$h : \PSh(F_p \times F_p\op) \to \PSh(Q_p \times Q_p\op)$$ be the induced essential geometric morphism. Then the (terminal-connected, \'etale) factorization and (pure, complete spread) factorization are given by
\begin{equation*}
\begin{tikzcd}
& \PSh(Q_p \times F_p\op) \ar[rd,"{\text{\'etale}}"] & \\
\PSh(F_p \times F_p\op) \ar[ru,"{\text{terminal-connected}}"] \ar[rd,"{\text{pure}}"'] & & \PSh(Q_p \times Q_p\op) \\
& \PSh(F_p \times Q_p\op) \ar[ru,"{\text{complete spread}}"'] & 
\end{tikzcd}
\end{equation*}
with each geometric morphism induced by the inclusion of submonoids. This gives an example of an essential geometric morphism where the (terminal-connected, \'etale) factorization and (pure, complete spread) factorization are both nontrivial and distinct from each other. 

To verify that the diagram above gives the correct (terminal-connected, \'etale) and (pure, complete spread) factorizations, we can either use the characterizations of Corollary \ref{cor:tc}, Theorem \ref{thm:etale}, Proposition \ref{prop:characterization-pure} and Theorem \ref{thm:complete-spread}, or use the following shortcut:

\begin{lemma}
Let $f : \PSh(M) \to \PSh(N)$ be the essential geometric morphism induced by a monoid map $\phi: M \to N$. For a monoid $P$, consider the monoid map $\phi_P : M \times P \to N \times P$ with $\phi_P(m,p) = (\phi(m),p)$. Let $f_P : \PSh(M \times P) \to \PSh(N \times P)$ be the geometric morphism induced by $\phi_P$. If $f$ is terminal-connected (resp.\ \'etale, pure, a complete spread), then $f_P$ is terminal-connected (resp.\ \'etale, pure, a complete spread) as well. 
\end{lemma}
\begin{proof}
It is enough to prove the statement for terminal-connected or \'etale geometric morphisms, the statement for pure geometric morphisms and complete spreads then follows by dualization. 

We can write $\PSh(N\times P) \simeq \PSh(N) \times \PSh(P)$, i.e.\ $\PSh(N\times P)$ is the product of $\PSh(N)$ and $\PSh(P)$ in the category of toposes. If $f : \PSh(M) \to \PSh(N)$ is \'etale, then so is $f_P : \PSh(M\times P) \to \PSh(N \times P)$, because \'etale geometric morphisms are stable under base change.

Now suppose that $f : \PSh(M) \to \PSh(N)$ is terminal-connected. Then $N$ is connected as left $M$-set. It follows that $N \times P$ is connected as left $(M \times P)$-set, in other words $f_P : \PSh(M\times P) \to \PSh(N\times P)$ is terminal-connected.
\end{proof}

\section{Application: Galois theory}
\label{sec:galois}

\subsection{Background on Galois theory for toposes}

For a locally connected topos $\Ecal$, there is a well-known notion of Galois theory built out of locally constant objects. We recall how this works below, following \cite{zoonekynd}, \cite{leroy}, \cite{bunge-funk-spreads-II}. 

First, observe that if $\{ A_i \}_{i \in I}$ is a family of locally constant objects, then in general the disjoint union $\bigsqcup_{i \in I} A_i$ need not be locally constant. For example, in the topos of continuous actions of the profinite integers, each action of the form $\Zbb/n\Zbb$ is locally constant and trivialized by itself, but the disjoint union $\bigsqcup_{n \in \Nbb} \Zbb/n\Zbb$ is not locally constant, since there is no object which can trivialize all of these cycles at once. If we consider the full subcategory $\mathbf{SLC}(\Ecal) \subseteq \Ecal$ consisting of the objects that are disjoint unions of locally constant objects, however, it turns out that $\mathbf{SLC}(\Ecal)$ is again a topos and the functor including $\mathbf{SLC}(\Ecal)$ into $\Ecal$ is the inverse image functor of a connected geometric morphism $g: \Ecal \to \mathbf{SLC}(\Ecal)$; see \cite[Th\'eor\`eme 2.4.(i)]{leroy}. Moreover, $\mathbf{SLC}(\Ecal)$ is a \textbf{Galois topos}, i.e.\ a 2-categorical cofiltered limit of toposes of the form $\PSh(G)$, with $G$ a (discrete) groupoid; see \cite[Theorem 1.6]{zoonekynd}.

We say that a locally connected topos $\Ecal$ is \textbf{locally simply connected} \cite{barr-diaconescu} if there exists a single family $\{U_k\}_{k \in K}$ jointly covering $1$ which trivializes each locally constant object in $\Ecal$. In this case, $\mathbf{SLC}(\Ecal) \simeq \PSh(G)$ for some (discrete) groupoid $G$. For example, let $X$ be a path-connected, locally path-connected, semilocally simply connected space. Take an open covering $\bigcup_{k \in K} U_k = X$ such that each $U_k$ is simply connected. Then, recalling that the covering maps over $X$ correspond to the locally constant objects in $\Sh(X)$, we see that $\{U_k\}_{k \in K}$ jointly covers the terminal object and trivializes the locally constant objects, so the topos $\Sh(X)$ is locally simply connected as one would hope. In this case, the covering maps $Y \to X$ correspond in $\mathbf{SLC}(\Ecal)$ to sets with a right action of the fundamental group $\pi_1(X)$, or in other words $\mathbf{SLC}(\Ecal) \simeq \PSh(\pi_1(X))$. 

We can also consider the small \'etale topos $\mathrm{Spec}(K)_\et$ associated to a field $K$. In this topos, every object is locally constant, so $\mathrm{Spec}(K)_\et \simeq \mathbf{SLC}(\mathrm{Spec}(K)_\et)$ is a Galois topos. More precisely, it is equivalent to the topos $\mathbf{Cont}(\mathrm{Gal}(K^s/K))$ of continuous right $\mathrm{Gal}(K^s/K)$-sets, with $\mathrm{Gal}(K^s/K)$ the absolute Galois group of $K$.

\subsection{Galois theory for toposes of monoid actions}

We now apply the concepts above to the topos $\PSh(N)$ for a monoid $N$.

For every small category $\Ccal$ there is a functor $\eta : \Ccal \to \Pi(\Ccal)$ to a groupoid $\Pi(\Ccal)$, unique up to equivalence, such that every functor from $\Ccal$ to a groupoid factors uniquely through $\eta$. Concretely, $\Pi(\Ccal)$ can be constructed as the groupoid with the same objects of $\Ccal$ in which morphisms are equivalence classes of composites of morphisms and formal inverses of morphisms in $\Ccal$. In the case that $\Ccal$ is a monoid $N$, this construction produces a group, that we will call the \textbf{groupification} and denote by $\pi_1(N)$; for $N$ commutative, $\pi_1(N)$ is known as the \textit{Grothendieck group} of $N$.

We can deduce from Proposition \ref{prop:lcpshf} that in the special case of presheaf toposes, coproducts of locally constant objects are locally constant, so $\mathbf{SLC}(\PSh(\Ccal))$ consists precisely of the locally constant objects, and moreover we can recover the well-known result that $\mathbf{SLC}(\PSh(\Ccal)) \simeq \PSh(\Pi(\Ccal))$, by observing that the locally constant presheaves of $\Ccal$ are precisely those which extend along $\eta : \Ccal \to \Pi(\Ccal)$. The connected geometric morphism $\PSh(\Ccal) \to \mathbf{SLC}(\PSh(\Ccal))$ then agrees with the essential geometric morphism induced by the functor $\eta: \Ccal \to \Pi(\Ccal)$. In particular, we shall in the remainder denote by $g : \Setswith{N} \to \Setswith{\pi_1(N)}$ the essential geometric morphism induced by the homomorphism $N \to \pi_1(N)$. The locally constant objects in $\Setswith{N}$ are precisely the objects of the form $g^*(X)$ for $X$ in $\Setswith{\pi_1(N)}$, and hence a geometric morphism with codomain $\Setswith{N}$ is locally constant \'etale if and only if it is of the form
\begin{equation*}
\begin{tikzcd}
\Setswith{N}/g^*(X) \ar[r] & \Setswith{N}
\end{tikzcd}
\end{equation*}
for $X$ in $\Setswith{\pi_1(N)}$. In light of the discussion above, the following result should not be unexpected.

\begin{corollary}
For any monoid $N$, $\Setswith{N}$ is a locally simply connected topos.
\end{corollary}
\begin{proof}
We show that $N$, as a right $N$-set, trivializes every locally constant object. Indeed, if $A$ is locally constant then by Proposition \ref{prop:lcpshf} $N$ acts bijectively on $A$, so the mapping
\begin{align*}
\coprod_{a \in A} N &\to A \times N \\
(a,n) &\mapsto (a \cdot n, n)
\end{align*}
is easily verified to be an isomorphism which commutes with the required maps.
\end{proof}

More generally, any connected presheaf topos is locally simply connected, see \cite[Corollary 7.9]{bunge-funk-spreads-II}. The proof there works for general presheaf toposes as well.

We can rephrase Corollary \ref{crly:lc-generators} in terms of $\pi_1(N)$.

\begin{theorem} \label{thm:classification-of-locally-constant-etale}
Let $N$ be a monoid and let $X$ be an object of $\Setswith{\pi_1(N)}$. Let $g : \Setswith{N} \to \Setswith{\pi_1(N)}$ be the geometric morphism induced by the groupification map $\eta : N \to \pi_1(N)$. Then the following are equivalent:
\begin{enumerate}
\item there is an equivalence $\Setswith{N}/g^*(X) \simeq \Setswith{M}$ for some monoid $M$;
\item there is a subgroup $H \subseteq \pi_1(N)$ such that $X \cong H \backslash \pi_1(N)$ and for all $y \in \pi_1(N)$ there is some $u \in \eta(N^\ltimes)$ such that $yu \in H$.
\end{enumerate}
In this case, $M = \eta^{-1}(H)$, and 
\begin{equation*}
\begin{tikzcd}
\Setswith{N}/g^*(X) \simeq \Setswith{M} \ar[r] & \Setswith{N}
\end{tikzcd}
\end{equation*}
agrees with the essential geometric morphism induced by the inclusion $M \subseteq N$.
\end{theorem}
\begin{proof}
We know from Theorem \ref{thm:etale-construction} that $\Setswith{N}/g^*(X) \simeq \Setswith{M}$ for some monoid $M$ if and only if $g^*(X)$ contains a strong generator, i.e.\ an element $x \in g^*(X)$ such that for all $y \in g^*(X)$ there is some $u \in N^\ltimes$ such that $yu = x$. This implies that $g^*(X)$ is connected as right $N$-set, so a fortiori it is connected as right $\pi_1(N)$-set. Because $\pi_1(N)$ is a group, we deduce that $X$ can be written as a quotient $X = H\backslash \pi_1(N)$, where $H \subseteq \pi_1(X)$ is the stabilizer of $x$. The condition that $x$ is a strong generator can be reformulated by saying that for any $y \in \pi_1(N)$ there is some $u \in \nu(N^\ltimes)$ such that $yu \in H$. Applying the formula from Theorem \ref{thm:etale-construction}, $M$ is then given by 
\begin{equation*}
N_x = \{ n \in N : xn = x \} = \eta^{-1}(H)
\end{equation*}
because $H$ is by definition the stabilizer of $x$ for the right $\pi_1(N)$-action.
\end{proof}

\begin{example}
Consider the monoid $N = \Zbb_p^\ns$ of nonzero $p$-adic integers under multiplication. The groupification of $\Zbb_p^\ns$ is the group $\Qbb_p^*$ of nonzero $p$-adic rational numbers. Consider the subgroup $H = \{ p^k : k \in \Zbb \} \subseteq \Qbb_p^*$. For all $g \in \Qbb_p^*$ there is an $u \in \Zbb_p^*$ such that $gu = p^k$ for some $k \in \Zbb$. So we are in the setting of Theorem \ref{thm:classification-of-locally-constant-etale}. We find $M = N \cap H = \{ p^k : k \in \Nbb \} \cong \Nbb$. So we see that the geometric morphism $\Setswith{\Nbb} \to \Setswith{\Zbb_p^\ns}$ induced by the inclusion $\Nbb \to \Zbb_p^\ns,~k \mapsto p^k$ is not only \'etale (as we already saw in Example \ref{eg:etale-geometric-morphisms}), but even locally constant \'etale. We can think of it as the covering map of $\PSh(\Zbb_p^\ns)$ corresponding to the subgroup $H \subseteq \Qbb_p^*$.
\end{example}

Observe that if we are given a presentation of a monoid, we can easily compute the groupification by interpreting the same presentation as a presentation of a group.

\begin{example}
Consider the \textit{bicyclic semigroup} $B$ with presentation 
\begin{equation*}
B = \langle{u,v: uv=1 }\rangle. 
\end{equation*}
Every element in $B$ can be written in a unique way as $v^iu^j$ for some $i,j \in \Nbb$. The right-invertible elements are $B^\ltimes = \{u^k : k \in \Nbb\}$ and the left-invertible elements are $B^\rtimes = \{ v^k : k \in \Nbb \}$.  We find $\pi_1(B) \simeq \langle{u,v : uv = 1}\rangle \simeq \langle{u}\rangle$. So we identify $\pi_1(B)$ with $\Zbb$, taking $u$ as the generator; the groupification map is $\eta: B \to \Zbb,~ v^iu^j \mapsto j-i$. The subgroups of $\Zbb$ are of the form $d\Zbb \subseteq \Zbb$ for $d \in \Nbb$. The equivalent conditions of Theorem \ref{thm:classification-of-locally-constant-etale} are satisfied if and only if $d \neq 0$, so for each $d \in \{1,2,3,\dots\}$ we get a locally constant \'etale geometric morphism 
\begin{equation*}
f_d : \Setswith{B_d} \to \Setswith{B} 
\end{equation*}
with $B_d = \eta^{-1}(d\Zbb) = \{ v^ju^i \in B : i\equiv j \mod{d} \}$. We borrowed the notation $B_d$ from \cite[(1.4)]{munn}, where it is proved that $B_d$ is a regular, simple semigroup whose idempotents form a submonoid isomorphic to $(\Nbb,\max)$.
\end{example}

\subsection*{Acknowledgements}

The first named author is a postdoctoral fellow of the Research Foundation – Flanders (file number 1276521N). Part of the present work was done while the second named author was Marie Sklodowska-Curie fellow of the Istituto Nazionale di Alta Matematica at Universit\`a degli Studi dell{'}Insubria, funded under the INDAM-DP-COFUND-2015 project.

We would like to thank the organisers of Toposes Online, where some of the results in this article were presented. 

\providecommand{\bysame}{\leavevmode\hbox to3em{\hrulefill}\thinspace}
\providecommand{\MR}{\relax\ifhmode\unskip\space\fi MR }
\providecommand{\MRhref}[2]{%
  \href{http://www.ams.org/mathscinet-getitem?mr=#1}{#2}
}
\providecommand{\href}[2]{#2}

\end{document}